\documentclass[11pt]{article}
\usepackage{amsthm}
\newtheorem{theorem}{Theorem}
\newtheorem{lemma}{Lemma}
\newtheorem{assumption}{Assumption}
\newtheorem{conjecture}{Conjecture}
\newtheorem{definition}{Definition}
\usepackage{amsmath,amssymb,amsfonts}
\usepackage{physics}
\usepackage{geometry}
\usepackage{hyperref}
\usepackage{fancyhdr}
\usepackage{titlesec}
\usepackage{enumitem}
\usepackage{graphicx}
\usepackage{xcolor}
\usepackage{caption}
\usepackage{tikz}
\usepackage{tikz}
\usetikzlibrary{arrows.meta, backgrounds, fit}
\usepackage{xcolor}
\usepackage{booktabs}
\usepackage{array}
\usepackage{enumitem}
\usetikzlibrary{arrows.meta, positioning, shapes, calc}

\title{\textbf{Morrey’s Derivation of Hydrodynamics from Statistical Mechanics: A Modern Exposition}}
\author{Yuyang Wang}

\begin{document}
\maketitle
\begin{abstract}
This paper is a modern review of C. B. Morrey's 1955 paper on the derivation of the equations of hydrodynamics from statistical mechanics. It has two main parts. First, assuming the existence of certain invariant $N$-particle phase distributions, Morrey derives formal balance laws for mass, momentum, and energy. Second, he attempts to construct a concrete sequence of such distributions by imposing cell-wise constraints that match the prescribed macroscopic fields $(\rho,u,e)$. From this construction, he obtains a Gibbs-type limiting form and finally the Euler equations. Compared to classical results, this regime employs the scaling law of $N\epsilon^3=O(1)$ and leads to a more flexible pressure $P_{eq}$. Our goal is to explain what Morrey is trying to do, why it's natural and where its most delicate points lie. We will rewrite the argument in a more transparent way, distinguishing carefully between formal derivations, physical input, and statements that still need justification.
\end{abstract}

\tableofcontents

\section*{Acknowledgments}
I would like to thank Professor Zaher Hani for suggesting
this project and for his guidance and advice throughout.
Our conversations on the mathematics of Morrey's paper were
essential to this work.
His research on Hilbert's sixth problem inspired my
interest in this area, and I am glad to have had the
chance to explore a small corner of it.

This thesis was written during my final year at the
University of Michigan, Ann Arbor. It was a good time.

\section{Introduction}\label{sec:setup}

One of the central problems in mathematical physics is to explain how the equations of fluid motion emerge from the dynamics of many interacting particles. On the macroscopic side stand continuum equations such as the compressible Euler system, written in terms of density, velocity, and thermodynamic variables. On the microscopic side stand Newton's equations for a large system of particles interacting through short-range forces. The difficulty is that a finite particle system carries mass, momentum, and energy through point particles, whereas the macroscopic description requires smooth fields in space and time. Morrey's paper \cite{Morrey1955} is an early and ambitious attempt to make this passage mathematically intelligible.

A standard way to bridge this gap is to introduce a probability distribution on phase space, so that singular particle observables are replaced by reduced densities governed by Liouville's equation and the BBGKY hierarchy. This viewpoint goes back to Gibbs and was further developed in the statistical theory of liquids and gases, especially in the work of Born and Green \cite{BornGreen1949}. Morrey's paper belongs to an early stage of the broader program of deriving continuum equations from particle dynamics, a program that later led to Lanford's derivation of Boltzmann, the Newton-to-kinetic theory, and kinetic-to-fluid limits for the Euler equations \cite{Lanford1975,GallagherSaintRaymondTexier2014,Caflisch1980,BardosGolseLevermore1991,DengHaniMa2025}.

The present paper does not claim that Morrey's derivation is fully rigorous
in the modern sense; one of our main purposes is precisely to determine where,
and why, it falls short. Taking Morrey's 1955 argument as our starting point,
we recast it in a more transparent language and modern notation, and --- more
substantively --- we separate it into its logically distinct components:
formal derivations, rigorously proved estimates, heuristic physical input, and
genuine analytical gaps. Several steps that are stated only implicitly or in
passing in the original are here isolated as explicit hypotheses, and we
identify the minimal set of unproved statements on which the final result
depends. In particular, we single out the role of the constrained phase-space
measures \eqref{3.3}, the ergodic-type assumption (Assumption~\ref{assump:ergodic})
used to relate microscopic motion to macroscopic averages, and the unresolved
steps in the passage from the finite-$N$ distributions to the limiting Gibbs
hierarchy, which we formulate as a small number of precise conjectures. The
outcome is both an exposition and an analysis: it makes the original
construction usable with current tools, and it locates the real mathematical
difficulties of this route to the fluid equations within a clearly delimited
set of open problems.

\subsection{Basic setup }
In this paper, we consider systems of $N$ identical point-mass particles, each of mass $m_N$, any two of which repel each other with a force 
\[
- m_N^2 \Phi_N'(r),
\] 
where $\Phi_N(r)$ is a potential depending only on the inter-particle distance $r$.  
(If $\Phi_N'(r) > 0$, the particles repel; if $\Phi_N'(r) < 0$, they attract.)  

Following Morrey, we impose the \textbf{scaling law}: 
\begin{equation}\label{1.1}\tag{1.1}
    m_N=D\epsilon^3,\; m_N\Phi_N(r)=K\Phi(\frac{r}{\epsilon}) \;\;\text{and}\;\;Nm_N= O(1).
\end{equation}
with $\epsilon\rightarrow 0$ as $N\rightarrow\infty$. $D$ and $K$ are constants. This is the regime in which the total mass is finite while the interaction range shrinks to the microscopic scale $\epsilon$.
Furthermore, in this paper,  we assume the following for our potential $\Phi$.
\begin{assumption}\label{Phi}
    \begin{equation}
\begin{aligned}
\text{(i)} &\lim_{r\rightarrow0^+}r^s\Phi(r)=+\infty,\;\;\;\lim_{r\rightarrow+\infty}r^s\Phi(r)=0\;\; \text{for some fixed $s>3$};\\
\text{(ii)} &\quad \Phi'(r)\ \text{is continuous for } r>0\ \text{ and } \lim_{r\to\infty} r^{s+1}\Phi'(r)=0;\\[-2pt]
\text{(iii)} &\quad \exists\,\delta>0\ \text{ such that }\ 
  \limsup_{r\to 0^+}\,\frac{|r\,\Phi'(r)|}{\Phi(r)}=\delta.
\end{aligned}
\label{1.2}\tag{1.2}
\end{equation}
\end{assumption}
\textbf{Remark :}An example is $\Phi(r)=1/r^s ,\; s>3$. From our scaling law, one have $D\epsilon^3\Phi_N(r)=\epsilon^s/r^s$, which is $\Phi_N(r)=\epsilon^{s-3}/r^s$. This means the interaction $\Phi_N$ is indeed a short range force.

Let 
$
x_i = \big(x_i^1(t), x_i^2(t), x_i^3(t)\big)
$ 
denote the position of the $i$-th particle. The equations of motion are now
\begin{equation*}\label{1.3}
\dot{x}_i^\alpha = u_i^\alpha, 
\qquad 
\dot{u}_i^\alpha = -\frac{\partial H_{NN}}{\partial x_i^\alpha}, 
\qquad \alpha = 1,2,3,
\tag{1.3}
\end{equation*}
where $u_i = (u_i^\alpha)$ is the velocity vector of the particle $i$ and interaction term is 
\[
 H_{N} = \frac{K}{2}\sum_{j \neq i}^N\Phi(\frac{|x_i - x_j|}{\epsilon}).
\]

\medskip
From \eqref{1.3}, it follows that the total momentum:
\[
m_N \sum_{i=1}^N u_i^\alpha,
\]
and the Hamiltonian:
\[
\tfrac{1}{2} m_N \sum_{i=1}^N |u_i|^2 + \tfrac{1}{2} K \sum_{j \neq i}^N \Phi(\frac{|x_i - x_j|}{\epsilon}),
\]
is conserved.

\subsubsection{Liouville's equation}
Now, assume at $t=0$, the N-particle density $\pi_{NN}(t=0;x_1,...x_N;u_1,...,u_N)$ is given, we will have\[
\pi_{NN}(t;x;u)=\pi_{NN}(0;\mathcal{H}^{-t}(x;u))
\]
where $\mathcal{H}^{-t}(x;u)$ is the time reversal of $(x,u)$ under flow \eqref{1.3}. 
Because the Hamiltonian flow \eqref{1.3} preserves the phase-space volume, the full density $\pi_{NN}$ should satisfy \textbf{Liouville's equation}
\begin{equation}\label{1.4}\tag{1.4}
    \frac{\partial \pi_{NN}}{\partial t}
+
\sum_{j=1}^N
\big(
u_j^\alpha\,\frac{\partial \pi_{NN}}{\partial x_j^\alpha}
-
\frac{\partial H_{NN}}{\partial x_j^\alpha}\,\frac{\partial \pi_{NN}}{\partial u_j^\alpha}
\big)
=0.
\end{equation}
By BBGKY hierarchy, the $l$-density distribution $\pi_{Nl}$ is 
\begin{equation}\label{1.5}\tag{1.5}
    \pi_{Nl}(t;\,x_1,u_1;\,\dots;\,x_l,u_l)
=
\int_{-\infty}^{\infty}
\pi_{NN}(t;\,x_1,u_1;\,\dots;\,x_{N},u_{N})\,dx_{l+1}du_{l+1}...dx_Ndu_N.
\end{equation}

Assuming derivatives may pass under the integral in \eqref{1.5}, successive integration of \eqref{1.4} yields
\begin{equation*}\label{1.6}
\frac{\partial\pi_{Nl}}{\partial t}
+
\sum_{j=1}^l
\big(
u_j^\alpha\,\frac{\partial\pi_{Nl}}{\partial x_j^\alpha}
-
\frac{\partial H_{Nl}}{\partial x^\alpha_j}\,\frac{\partial\pi_{Nl}}{\partial u_j^\alpha}
\big)
=
-\frac{(N-l)K}{\epsilon}
\int_{-\infty}^{\infty}
\sum_{j=1}^l
\Phi'\!\big(\frac{|x_j-x_{l+1}|}{\epsilon}\big)\,
\frac{x_{l+1}^\alpha-x_j^\alpha}{|x_{l+1}-x_j|}
\;
\frac{\partial\pi_{N,l+1}}{\partial u_j^\alpha}\,dx_{l+1}\,du_{l+1},
\tag{1.6}
\end{equation*}
for $l=1,\dots,N-1$, with
\[
H_{Nl} \;=\; \frac{K}{2}\sum_{j\neq k}^l \Phi\!\big(\frac{|x_j-x_k|}{\epsilon}\big).
\]
\textbf{Remark :} The $l$-particle equation couples to the $(l+1)$-particle marginal through the interaction force. In the computation, the RHS are all contributed to integrating the term $H_{N(l+1),x_j^\alpha} \pi_{N(l+1),u_j^\alpha}$, $j<l+1$. Other terms vanish.

For a smooth observable $g$ on phase space, let $T_\tau$ denote the
Hamiltonian flow map of \eqref{1.3} for time $\tau$. We write
$\mu(\,\cdot\,)$ for \emph{integration against the measure} $\mu$, i.e.
\[
\mu(g):=\int_{\mathcal X} g\,d\mu .
\]
Two dual operators act here. On observables, $T_\tau^{*}$ is the pull-back
(Koopman) operator
\[
(T_\tau^{*}g)(x):=g(T_\tau x)=g\circ T_\tau ,
\]
while on measures $T_\tau$ is the push-forward (transfer) operator, defined
so that the change-of-variables identity holds:
\begin{equation}\label{duality}\tag{1.7}
(T_\tau\mu)(g)=\int g\,d(T_\tau\mu)=\int (g\circ T_\tau)\,d\mu=\mu(T_\tau^{*}g).
\end{equation}
Thus $\mu(T_\tau^{*}g)$ is a \emph{number}: the integral of the pulled-back
observable $T_\tau^{*}g$ against $\mu$, equivalently the integral of $g$
against the pushed-forward measure $T_\tau\mu$. (The bracket $\mu(\,\cdot\,)$
is always the integration pairing, never itself a push-forward; the
push-forward is the dual object $T_\tau\mu$.)

The \emph{Liouville operator} $\mathcal L$ is the generator of the pull-back
semigroup $T_\tau^{*}$ acting on observables \cite{CIP1994}:
\begin{equation}\label{Liouvilleop}\tag{1.8}
\mathcal L g
:=\frac{d}{d\tau}\Big|_{\tau=0}\!\big(T_\tau^{*}g\big)
=\{g,H_{NN}\}
=\sum_{j=1}^N\left(u_j^\alpha\,\frac{\partial g}{\partial x_j^\alpha}
-\frac{\partial H_{NN}}{\partial x_j^\alpha}\,\frac{\partial g}{\partial u_j^\alpha}\right),
\end{equation}
so that $\dfrac{d}{d\tau}\,T_\tau^{*}g=T_\tau^{*}\mathcal L g$; this is exactly
the operator and identity appearing in \eqref{1.10}. By duality the phase-space
density evolves by the formal adjoint $-\mathcal L$, which is precisely
Liouville's equation \eqref{1.4}:
\[
\partial_t\pi_{NN}=-\mathcal L\,\pi_{NN}
=-\sum_{j=1}^N\left(u_j^\alpha\,\frac{\partial \pi_{NN}}{\partial x_j^\alpha}
-\frac{\partial H_{NN}}{\partial x_j^\alpha}\,\frac{\partial \pi_{NN}}{\partial u_j^\alpha}\right).
\]

\subsubsection{Gibbs distribution}
\medskip

We close the setup by recording the equilibrium measures to which the
macroscopic fields will be referred, since the same Gibbs measure appears both
in the local-equilibrium distribution (Definition~\ref{def:localequi}) below and in Section~\ref{Section:limit}. 

Given macroscopic data
$(\rho,\bar u,e)$ --- a density $\rho>0$, a mean velocity $\bar u\in\mathbb R^3$,
and a specific total energy $e$ --- the associated \emph{Gibbs measure}
$\nu_{(\rho,\bar u,e)}$ is the equilibrium probability measure on phase space
whose $l$-particle correlation functions take the form
\begin{align*}
&\nu^{(l)}_{(\rho,\bar u,e)}(x_1,u_1;\dots;x_l,u_l)
\\=&
\rho^{\,l}\Big(\tfrac{A}{\pi}\Big)^{3l/2}
\exp\!\Big\{-A\!\sum_{i=1}^{l}\|u_i-\bar u\|^2\Big\}\,
\exp\!\Big\{-2A\,K\!\!\sum_{\substack{i,j=1\\ i< j}}^{l}\Phi\big(|x_i-x_j|/\epsilon\big)\Big\}\,
g_l(x_1,\dots,x_l),\label{gibbs}\tag{1.9}
\end{align*}
that is, a Maxwellian in the velocities centered at the mean velocity $\bar u$,
a Boltzmann weight $e^{-A\,K\sum_{i<j}\Phi}$ in the interaction energy, and a
spatial correlation factor $g_l$ encoding the deviation from statistical
independence (so that $g_l\to1$ when all particles are infinitely separated).
Here the scalar $A>0$ and the correlation $g_l$ are \emph{not} free: they are
the quantities conjugate to $(\rho,\bar u,e)$, fixed by the requirement that the
mean density, momentum, and energy computed from $\nu_{(\rho,\bar u,e)}$ equal
the prescribed fields. Concretely, $\bar u$ is the mean velocity, $\rho$ is the
mean particle density, and $A$ is determined by the internal energy through the
equipartition/virial relation made explicit in Section~\ref{sectioneuler}
(equations \eqref{eos-eps}--\eqref{eos-C}); this is the equation of state. The
construction and the characterization of such equilibrium measures via the
Dobrushin--Lanford--Ruelle equations are classical, see
\cite{Ruelle1969,BornGreen1949}.

We write $\nu_{G(t,x)}$ for the Gibbs measure whose parameters realize the local
fields $G(t,x)=(\rho,\bar u,e)(t,x)$; this is the measure appearing in
Definition~\ref{def:localequi}. The functions $\phi_l$ of the present paper are by
definition the $l$-particle densities of this Gibbs measure, written in the
local (relative) coordinates \eqref{2.3}: substituting $u_i=u+v_{i-1}$ and
$x_i=x+\epsilon\xi_{i-1}$ into \eqref{gibbs} and recalling
$2\mathcal X_l=K\sum_{i\neq k}\Phi(|\xi_i-\xi_k|)$ from \eqref{2.6} gives exactly
the limiting form
\[
\phi_l=(A/\pi)^{3l/2}\exp\!\Big\{-A\Big[\textstyle\sum_{i=0}^{l-1}\|u+v_i-\bar u\|^2+2\mathcal X_l\Big]\Big\}\,
\rho^l\,g_l(\xi_1,\dots,\xi_{l-1};\rho,A),
\]
which is precisely \eqref{6.1}. Thus the entire construction of
Sections~\ref{section3}--\ref{Sectiongl} may be read as a single statement: the
constrained marginals $\phi_{Nl}$ built from the cell data $(\rho,\bar u,e)$
converge, in the hydrodynamic limit, to the $l$-particle densities $\phi_l$ of
the Gibbs measure $\nu_{(\rho,\bar u,e)}$, with the parameter $A$ and the
correlations $g_l$ determined self-consistently by $(\rho,\bar u,e)$.

\subsection{Background motivation} \label{subsec:motivation}
To motivate the form of the limiting equations, we begin with a related result from \cite[Section 0]{DeMasiEtAl1984}, which clarifies the type of local equilibrium structure that should lead to the Euler equations. Let $\mathcal{X}$ be the abstract particle phase space. $P\in\mathcal{X}$ is a locally-finite microscopic configuration.

\begin{definition}[Local equilibrium distributions]\label{def:localequi}
A \emph{local equilibrium distribution} is a family $\mu^\epsilon$ of probability
measures on $\mathcal X$ such that:

\begin{enumerate}
\item[(i)] For every $t > 0$, the time evolution $T_\tau$ is defined for times
$\tau \le \epsilon^{-1}t$ on a set of full $\mu^\epsilon$-measure,
for $\epsilon \in (0,1]$. 

If $g$ is a smooth cylindrical function and $\mathcal L$ i
s the Liouville operator, then
\[\label{1.10}\tag{1.10}
\forall\, \tau \le \epsilon^{-1}t,
\qquad
\frac{d}{d\tau}\,\mu^\epsilon(T_\tau^* g)
=
-\mu^\epsilon(T_\tau^* \mathcal L g),
\]
where
\[
(T_\tau^* g)(x) = g(T_\tau x)
\quad \mu^\epsilon\text{-a.s.}
\]

\item[(ii)] There exists a continuous function
\[
G(t,x) : \mathbb R_+ \times \mathbb R^d 
\longrightarrow 
\mathbb R \times \mathbb R^d \times \mathbb R_+,
\qquad
G(t,x) = (\rho(t,x), u(t,x), e(t,x)),
\]
such that for all $x \in \mathbb R^d$, $t \ge 0$,
\[
\text{(weak)} \quad
\lim_{\epsilon \to 0}
D_{\epsilon^{-1}x} T_{\epsilon^{-1}t} \mu^\epsilon
=
\nu_{G(t,x)},\label{1.11}\tag{1.11}
\]
where
\[
(T_\tau \mu^\epsilon)(g) = \mu^\epsilon(T_\tau^* g),
\qquad
(D_x \mu^\epsilon)(g) = \mu^\epsilon(D_x^* g),
\]
and $D_x^* g$ denotes the translate of $g$ by $x$.
Here $\nu_{G(t,x)}$ is the \textbf{Gibbs measure} corresponding to the
parameters $G(t,x)$.

\item[(iii)] Let $\mathcal N(x)$ denote the number of particles and
$\iota(x)$ the kinetic energy of those particles contained in the
unit ball centered at $x$. Then for every $t > 0$,
\[
\sup_{\epsilon}
\sup_{x \in \mathbb Z^d}
\sup_{\tau \le \epsilon^{-1}t}
T_\tau \mu^\epsilon\big( \mathcal N(x)^2 + \iota(x)^2 \big)
< +\infty.\label{1.12}\tag{1.12}
\]
\end{enumerate}

The function $G(t,x)$ is called the \emph{equilibrium profile}
of the local equilibrium distribution $\mu^\epsilon$.
\end{definition}

Then, \cite[Section 0]{DeMasiEtAl1984} tells us that such a local equilibrium distribution leads to the Euler equations for the corresponding macroscopic parameters.

\begin{theorem}[Theorem 1.1]
Assume that $\mu^\epsilon$ is a local equilibrium distribution
with parameters
\[
G(t,x) = (\rho(t,x), u(t,x), e(t,x))
\quad.
\]
Then the parameters satisfy the Euler equations
(for notational simplicity we set $m=1$):
\begin{align}
\partial_\tau \rho&= - \nabla \cdot (\rho v), \tag{1.13a} \\
\partial_\tau u &= - (u \cdot \nabla)u 
- \frac{1}{\rho}\nabla P_{\mathrm{eq}}, \tag{1.13b} \\
\partial_\tau e &= - (u \cdot \nabla)e
- P_{\mathrm{eq}} (\nabla \cdot u)
- e (\nabla \cdot u), \tag{1.13c}
\end{align}
where $P_{\mathrm{eq}}(\xi,\tau)$ is the thermodynamic pressure
corresponding to the equilibrium parameters
$\rho(\xi,\tau), v(\xi,\tau), e(\xi,\tau)$.
\end{theorem}

However, the assumption that such a local equilibrium distribution exists is strong and hard to prove. From this perspective, Morrey’s program may be viewed as an attempt to construct a family of locally equilibrated states. In this paper, the measure $\mu^\epsilon$ is defined in equation \eqref{3.3} up to a combinatorial factor. From the point of view of the definition \ref{def:localequi}, $(ii)$ is partially proved in Section \eqref{Section:limit} and $(i)$ requires the ergodic assumption and will be discussed in Section \ref{sec:ergodic}

\subsection{What's the difference between Morrey's program and the results based on Boltzmann-Grad Limit}
It is worth contrasting the \emph{equation of state} obtained this way with
the one that arises in the classical analytic program around Hilbert's sixth
problem. In the kinetic route to hydrodynamics --- deriving the compressible
Euler (or Navier--Stokes--Fourier) equations from Newtonian dynamics through
Boltzmann's equation \cite{BardosGolseLevermore1991, GallagherSaintRaymondTexier2014,
DengHaniMa2025} --- one works in the dilute, Boltzmann--Grad regime of a
hard-sphere (or short-range) gas. The resulting fluid is therefore an
\emph{ideal monatomic gas}, with the rigidly prescribed pressure law
\[
P_{\mathrm{eq}} = (\gamma-1)\,\rho\, e_{\mathrm{int}},
\qquad \gamma = \tfrac{5}{3}\ \text{fixed},
\]
where $e_{\mathrm{int}}$ is the specific internal energy and the adiabatic
exponent $\gamma$ (the ratio of specific heats) is a \emph{fixed} constant
determined by the kinetic model, with no freedom left in the closure. This
restriction to the ideal-gas equation of state is intrinsic to the dilute
limit: it is precisely what confines those derivations to gases rather than
to liquids or dense, high-pressure fluids \cite{GallagherSaintRaymondTexier2014,
DengHaniMa2025}.

Morrey's construction is, by contrast, more flexible on exactly this point.
Because the closure is extracted from the limiting particle distribution
itself rather than from a dilute-gas collision kernel, the resulting equation
of state is governed by certain normalization function $C(\rho,b)$ of the
equilibrium family in Theorem \ref{Thm:limit} and is \emph{not} forced into the single-parameter
polytropic form $P_{\mathrm{eq}}=(\gamma-1)\rho\, e_{\mathrm{int}}$ with a
fixed $\gamma$. Instead the pressure is determined by the interaction
potential $\Phi$ through $C$, and reduces to an explicit thermodynamic
relation only after the limiting correlations are computed. We return to this
point in Section~\ref{sectioneuler}, where the equation of state is written
directly in terms of $C(\rho,b)$ and is seen to carry genuine thermodynamic
information rather than a prescribed exponent $\gamma$.

\textbf{In conclusion}, a significant difference between the classical program  \cite{BardosGolseLevermore1991, GallagherSaintRaymondTexier2014,
DengHaniMa2025} and Morrey's paper \cite{Morrey1955} is that: Instead of the Boltzmann-Grad limit which is $N\epsilon^{d-1}=O(1)$, Morrey employed the scaling law \eqref{1.1} $N\epsilon^3=O(1)$. It stands for the non-dilute, liquid state. In fact, given the short range assumption for the potential $\Phi$ \eqref{1.2}, the extra $\epsilon$ erases the influence of the time derivative $\frac{\partial \phi_{Nl}}{\partial t}$ in \eqref{2.5}, which leads to the equilibrium Gibbs distribution, instead of Boltzmann distribution. And such a regime leads to a more flexible choice of pressure $P_{eq}$.

\subsection{Structure of the paper}
First, in Section \ref{section2}, Morrey derived the $N$-particle balance laws by assuming the existence of an $N$-particle distribution $\phi_{NN}$. This is a standard process.

Then, in Section \ref{section3} and \ref{Section:phiNL}, Morrey constructed a concrete sequence of measures $\phi_{Nl}$ corresponding to certain $(\rho,u,e)$. The idea is to select the $N$-particle systems that have a macro-state similar to $(\rho,u,e)$. In particular, we split the $x$-space $\mathbb{R}^3$ into small cells $R_N$ and want the mass, momentum, and energy of an $N$-particle system to have the same values given by $(\rho,u,e)$. For example, in each cell $R_N$, we require
\[
m_N\#\{x_i\in R_N\}=\int_{R_N}\rho(t,x)\,dx.
\]
Such restrictions give us a sub-manifold $M_N$ in the phase space $\mathbb{R}^{6N}$ and an invariant probability measure on it induced by the Lebesgue measure, which is our $\phi_{NN}$.

After that, in Section \ref{Section:limit}, Morrey tries to prove that such $\phi_{Nl}$ given by the BBGKY hierarchy eventually converge to a Gibbs measure $\phi_l$ in \eqref{6.1}, which actually stands for the convergence in \eqref{1.11}. Formally, if this convergence is strong enough and if the ergodic-type assumption in Assumption \ref{assump:ergodic} is valid, then the balance laws from Section \ref{section2} should pass to the limit and yield the Euler equations.

However, in Section \ref{section3}, we will question whether $\phi_{Nl}$ is indeed invariant along the Newton flow \eqref{1.3}. This is because the particle flows are not exactly moving along the restricted submanifold $M_N$, and $\phi_{Nl}$ is designed to be invariant under the flow projected on $M_N$, not under the Newton flow itself. Thus, the equations derived in Section \ref{section2} will only hold up to some errors. Fortunately, we will discuss in Section \ref{Fourier} that the Newton flow orthogonal to $M_N$ is significantly smaller than the flow along $M_N$, which means that $\phi_{Nl}$ is still almost flow-invariant. Hence, it is still plausible to expect that such $\phi_{Nl}$ will serve as the local equilibrium distribution.

Based on the above discussion, there are still the ergodic assumption \ref{assump:ergodic} and three conjectures \ref{conj:pi}, \ref{conj:cell}, and \ref{conj:deri} which remain unsolved. This is the main gap in Morrey's derivation. By assuming them, we will obtain the hydrodynamic equations in Section \ref{sectioneuler}. 

Through out the paper, one can see that:  The first part of the paper, Section \ref{section2}, is standard and rigorous. The most important and non-rigorous part is the construction of the marginals and their limit: Section \ref{section3}-\ref{Section:limit}.

\section{Particle Distribution  and balance laws}\label{section2}
This section derives the formal balance laws associated with an invariant N-particle phase-space distribution $\pi_{NN}$. These equations already resemble hydrodynamics, but they are not closed, since they still depend on higher-order marginals.

We proceed in two steps. First, Liouville's equation for the full $N$-particle density is integrated to obtain the BBGKY hierarchy for the reduced marginals, which was done in the setup section. Second, the marginals are renormalized so that mass, momentum, and energy in finite cells remain of order one under Morrey's scaling. Third, we pass to local coordinates adapted to the interaction scale $\varepsilon$, which isolates the short-range structure of the hierarchy. This is a standard and rigorous derivation.

\subsection{Renormalization and local variables}
To pass to the large-$N$ limit while keeping finite-cell quantities well-defined, we normalize and define the mass distribution
\[\label{2.1}
p_{Nl}(t;x_1,...,x_l;u_1,...,u_l) 
= \frac{m_N^l N!}{ (N-l)!}\,\pi_{Nl}(t;x_1,...,x_l;u_1,...,u_l).
\tag{2.1}
\]

Combining this with \eqref{1.6}, we will have
\begin{equation*}\label{2.2}
p_{Nl,t}
+
\sum_{j=1}^l
\big(
u_j^\alpha\,p_{Nl,x_j^\alpha}
-
H_{Nl,x_j^\alpha}\,p_{Nl,u_j^\alpha}
\big)
=
-\frac{K}{D\epsilon^4}
\int_{-\infty}^{\infty}
\sum_{j=1}^l
\Phi'\!\big(|x_j-x_{l+1}|/\epsilon\big)\,
\frac{x_{l+1}^\alpha-x_j^\alpha}{|x_{l+1}-x_j|}
\;
p_{N,l+1}\,dx_{l+1}\,du_{l+1}.
\tag{2.2}
\end{equation*}
\textbf{Remark:} The $1/\epsilon^4=1/m_N\epsilon$ enters from differentiating $\Phi_N(r/\epsilon)$ get one $1/\epsilon$ and $p_{N,l+1}$ eats a $m_N$; this is the rescaled force.

To exploit the short-range scaling, we introduce relative variables
\begin{equation*}\tag{2.3}\label{2.3}
x=x_1,\quad u=u_1,\qquad
\xi_j=\frac{x_{j+1}-x}{\epsilon},\quad
v_j=u_{j+1}-u,\qquad j=1,\dots,\,l-1,
\end{equation*}
and define functions $\phi_{N1},\dots,\phi_{Nl}$ by
\[\tag{2.4}\label{2.4}
\phi_{Nl}(t;\,x,u;\,\xi_1,v_1;\,\dots;\,\xi_{l-1},v_{l-1})
=
p_{Nl}\big(t;\,x_1,u_1;\,x+\epsilon\xi_1,u+v_1;\,\dots;\,x+\epsilon\xi_{l-1},u+v_{l-1}\big).
\]
This centers the first particle at $(x,u)$ and expresses others by $O(\epsilon)$ offsets $(\xi,v)$, isolating the local interaction scale.
$\xi_j=(x_{j+1}-x_1)/\epsilon$

\textbf{Remark:} We introduced three particle distributions: $\pi_{Nl}$ is the raw marginal density, $p_{Nl}$ is normalized mass density and $\phi_{Nl}$ is the rescaled local coordinate version.

With the change of variables \eqref{2.4}  and after multiplying \eqref{2.2} by $\epsilon$, we obtain the system
\begin{equation*}\label{2.5}
\begin{aligned}
\epsilon\big(\phi_{Nl,t}+u^\alpha \phi_{Nl,x^\alpha}\big)
\;&+\; B_{l0}^\alpha(\xi)\,\phi_{Nl,u^\alpha}
\;+\; \sum_{i=1}^{\,l-1}\big[\,u^\alpha \phi_{Nl,\xi_i^\alpha}+B_{li}^\alpha(\xi)\,\phi_{Nl,v_i^\alpha}\big]
\end{aligned}
\tag{2.5}
\end{equation*}
\begin{equation*}
    =
-\frac{K}{D}\int_{-\infty}^{\infty}
\bigg\{
\Phi_N'\big(|\xi_1|\big)\,\frac{\xi_1^\alpha}{|\xi_1|}
\;\phi_{N,l+1,\;u^\alpha}
\;+\;
\sum_{i=1}^{\,l-1}
\bigg[
\Phi_N'\big(|\xi_i-\xi_1|\big)\,\frac{\xi_i^\alpha-\xi_1^\alpha}{|\xi_i-\xi_1|}
-
\Phi_N'\big(|\xi_1|\big)\,\frac{\xi_1^\alpha}{|\xi_1|}
\bigg]\,
\phi_{N,l+1,\;v_i^\alpha}
\bigg\}\,d\xi_1\,dv_1,
\end{equation*}
with coefficients
\begin{equation*}
B_{l0}^\alpha(\xi)=K\sum_{j=1}^{\,l-1}\Phi_N'\big(|\xi_j|\big)\,\frac{\xi_j^\alpha}{|\xi_j|},
\qquad
B_{li}^\alpha(\xi)
=-K\left[\Phi_N'\big(|\xi_i|\big)\,\frac{\xi_i^\alpha}{|\xi_i|}
+\sum_{k=1}^{\,l-1}\Phi_N'\big(|\xi_i-\xi_k|\big)\,\frac{\xi_i^\alpha-\xi_k^\alpha}{|\xi_i-\xi_k|}\right].
\end{equation*}

Moreover,
\[\tag{2.6}\label{2.6}
B_{l0}^\alpha+\sum_{i=1}^{\,l-1}B_{li}^\alpha=0,
\qquad
B_{l0}^\alpha+B_{li}^\alpha=-\mathcal{X}_{l,\xi_i^\alpha},
\qquad
\mathcal{X}_l=\frac{K}{2}\sum_{j,k=0,j\neq k}^{\,l-1}\Phi\big(|\xi_j-\xi_k|\big),
\quad (\xi_0=0).
\]
In these variables, the distinguished particle sits at $(x,u)$, while the remaining particles are described relative to it on the interaction scale $\epsilon$. The macroscopic variables t and x now appear only through slow variations, while the fast interaction terms are encoded in the $(\xi,v)$ variables.

\textbf{Remark:} These identities express momentum conservation among relative coordinates; $\mathcal{X}_l$ is the interaction energy in the relative frame. To check the computation, notice:
\begin{equation*}
    \frac{\partial \phi_{N,N}}{\partial x}=\frac{\partial p_{NN}}{\partial x_1}+...+\frac{\partial p_{NN}}{\partial x_N}\;\;\text{and}\;\; \frac{\partial \phi_{N,N}}{\partial \xi_j}=\epsilon\frac{\partial p_{NN}}{\partial x_{j+1}}
\end{equation*}

\medskip
From the definitions above, the cell-wise densities take the continuous forms
\begin{equation*}\label{2.7}
\begin{aligned}
\rho_N(t;x) \;&=\; \int_{-\infty}^{\infty}\phi_{N1}(t;x,u)\,du,\\
\rho_N(t;x)\,\bar{u}_N^\alpha(t;x) \;&=\; \int_{-\infty}^{\infty}u^\alpha\,\phi_{N1}(t;x,u)\,du,\\
\rho_N(t;x)[\frac{1}{2}\|\bar{u}_N(t;x)\|^2+\varepsilon_N(t;x)]:&=e_N(t;x) \;&
\\&=\; \frac12\int_{-\infty}^{\infty}|u|^2\,\phi_{N1}(t;x,u)\,du
\;+\;\frac{K}{2D}\int_{-\infty}^{\infty}\Phi\big(|\xi_1|\big)\,
\phi_{N2}(t;x,u;\xi_1,v_1)\,du\,d\xi_1\,dv_1.
\end{aligned}
\tag{2.7}
\end{equation*}
\textbf{Remark :} In this paper $e$ stands for \emph{total energy} and $\varepsilon$ stands for internal energy.
\subsection{Balance equations}
The equation \eqref{2.5} for $l=1$ can be written as
\begin{equation*}\label{2.8}
\phi_{N1,t}+u^\alpha\phi_{N1,x^\alpha} = -\,f_{N1,u^\alpha}^\alpha,
\tag{2.8}
\end{equation*}
with\[
f_{N1}^\alpha(t;x,u)
:=
\frac{K}{D\epsilon}
\int_{-\infty}^{\infty}
\Phi'\big(|\xi_1|\big)\,\frac{\xi_1^\alpha}{|\xi_1|}
\;\phi_{N2}(t;x,u;\xi_1,v_1)\,d\xi_1\,dv_1.
\]
Integrating \eqref{2.8} with respect to $u$ gives \emph{continuity}:
\begin{equation*}\label{2.9}
\frac{\partial \rho_N(t;x)}{\partial t}+\frac{\partial}{\partial x^\alpha}\!\big[\rho_N(t;x)\,\bar{u}_N^\alpha(t;x)\big]=0.
\tag{2.9}
\end{equation*}

\textbf{Remark:} \eqref{2.8} plays an important role in the following derivation. It is force density in space.  After integration, RHS vanishes by parts (decay at $|u|\to\infty$). It is reasonable since we want to avoid having too many fast particles. LHS= mass density change in time+ momentum density change in space=0 are conserved.

Multiplying \eqref{2.8} by $u^\beta$, integrating in $u$, and integrating by parts on the right (assuming $u^\beta f_{N1}^\alpha\!\to 0$ as $|u|\to\infty$), we obtain the \emph{momentum equation}
\begin{equation*}\label{2.10}
\frac{\partial}{\partial t}\big(\rho_N\bar u_N^\beta\big)
+\frac{\partial}{\partial x^\alpha}
\big(\rho_N\bar u_N^\alpha \bar u_N^\beta+\psi_{N1}^{\alpha\beta}\big)
=F_{N1}^\beta(t;x),
\tag{2.10}
\end{equation*}
with
\[
\psi_{N1}^{\alpha\beta}(t;x):=\int_{-\infty}^{\infty}(u^\alpha-\bar u_N^\alpha)(u^\beta-\bar u_N^\beta)\,\phi_{N1}(t;x,u)\,du,
\]
and force density in space
\[
F_{N1}^\beta(t;x)=\int_{-\infty}^{\infty} f_{N1}^\beta(t;x,u)\,du
=
\frac{K}{D\,\epsilon}
\int_{-\infty}^{\infty}\Phi'\big(|\xi_1|\big)\,\frac{\xi_1^\beta}{|\xi_1|}
\;\phi_{N2}(t;x,u;\xi_1,v_1)\,du\,d\xi_1\,dv_1.
\]
\textbf{Remark:} Integration by parts in $u$ gives $-\!\int u^\beta \partial_{u^\alpha}f^\alpha\,du=\int f^\beta\,du$ via $\partial_{u^\alpha}u^\beta=\delta^{\beta\alpha}$.

Multiplying \eqref{2.8} by $\tfrac12|u|^2$ and arguing similarly yields the kinetic-energy balance
\begin{equation*}\label{2.11}
\frac{\partial}{\partial t}\!\left[\tfrac12\int_{-\infty}^{\infty}|u|^2\,\phi_{N1}\,du\right]
+
\frac{\partial}{\partial x^\alpha}\!\left[\tfrac12\int_{-\infty}^{\infty}u^\alpha |u|^2\,\phi_{N1}\,du\right]
=
\frac{K}{D\,\epsilon}
\int_{-\infty}^{\infty}
\Phi'\big(|\xi_1|\big)\,\frac{\xi_1^\alpha}{|\xi_1|}\,
u^\alpha \,\phi_{N2}(t;x,u;\xi_1,v_1)\,du\,d\xi_1\,dv_1.
\tag{2.11}
\end{equation*}

From the symmetry of $p_{N2}$ under particle exchange,
\begin{equation*}\label{2.12}
\begin{aligned}
\phi_{N2}(t;x,u;\xi_1,v_1)
&=p_{N2}(t;x,u;\,x+\epsilon\xi_1,u+v_1)
= p_{N2}(t;x+\epsilon\xi_1,u+v_1;\,x,u)
\\
&=\phi_{N2}(t;x+\epsilon\xi_1,u+v_1;\,-\xi_1,-v_1)
=\phi_{N2}(t;x-\epsilon\xi_1',u';\,\xi'_1,v'_1),
\end{aligned}
\tag{2.12}
\end{equation*}
where $\xi'_1=-\xi_1$, $v'_1=-v_1$, and $u'=u+v_1$.
Consequently,
\begin{equation*}\label{2.13}
\big(u^\beta+\tfrac12 v_1^\beta\big)\,\phi_{N2}(t;x,u;\xi_1,v_1)
=
\big(u^\beta+\tfrac12 v_1^{\prime\beta}\big)\,\phi_{N2}(t;x-\epsilon\xi_1',u';\,\xi'_1,v'_1).
\tag{2.13}
\end{equation*}
In particular, we may write
\begin{equation*}\label{2.14}
\phi_{N2}(t;x,u;\xi_1,v_1)
=
\tfrac12\,\phi_{N2}(t;x,u;\xi_1,v_1)
+\tfrac12\,\phi_{N2}(t;x-\epsilon\xi_1',u';\,\xi'_1,v'_1).
\tag{2.14}
\end{equation*}
Using \eqref{2.14} in the definition of $F_{N1}$\eqref{2.10} gives the difference–quotient form
\begin{equation*}\label{2.15}
F_{N1}^\beta(t;x)
=
\frac{K}{2D}
\int_{-\infty}^{\infty}
\Phi'\big(|\xi_1|\big)\,\frac{\xi_1^\beta}{|\xi_1|}
\left[
\frac{\phi_{N2}(t;x,u;\xi_1,v_1)-\phi_{N2}(t;x-\epsilon\xi_1,u';\,\xi_1,v_1)}{\epsilon}
\right]\,du\,d\xi_1\,dv_1.
\tag{2.15}
\end{equation*}
This stresses that the force is governed by a derivative related to the space ($x$) of the 2-distribution.

\textbf{\textbf{Remark:}} In \eqref{2.14}, the plus become minus in \eqref{2.15} since the force get reversed! Let\[
K^\beta(\xi_1) := \Phi'(|\xi_1|)\frac{\xi_1^\beta}{|\xi_1|},
\qquad K^\beta(-\xi_1) = -K^\beta(\xi_1).
\]
\textbf{symmetricity}  
 $(\xi_1,v_1)\mapsto(-\xi_1,-v_1)$ 
\begin{equation*}
\begin{aligned}
F_{N1}^\beta(t;x)
&=\frac{K}{2D\,\epsilon}\int \Big[
K^\beta(\xi_1)\,\phi_{N2}(t,x,u;\xi_1,v_1)
+K^\beta(-\xi_1)\,\phi_{N2}(t,x,u;-\xi_1,-v_1)
\Big]\,du\,d\xi_1\,dv_1.
\end{aligned}
\end{equation*}
Then use \eqref{2.12} for $\phi_{N2}(t,x,u;-\xi_1,-v_1)$ \[
\phi_{N2}(t,x,u;-\xi_1,-v_1)=\phi_{N2}(t,x-\epsilon\xi_1,u-v_1;\xi_1,v_1)
\]
Then we get our \eqref{2.15} where now $u'=u-v_1$.

The equation of system \eqref{2.5} that involves $l=2$ can be written schematically as
\begin{equation*}\label{2.16}
\phi_{N2,t}+u^\alpha\phi_{N2,x^\alpha}
+\epsilon^{-1}\Big[
K\Phi'\big(|\xi_1|\big)\frac{\xi_1^\alpha}{|\xi_1|}\,\phi_{N2,u^\alpha}
+v_1^\beta\,\phi_{N2,v_1^\beta}
\Big]
-2K\Phi'\big(|\xi_1|\big)\frac{\xi_1^\alpha}{|\xi_1|}\,\phi_{N2,u^\alpha}
=
F_{N20}-F_{N20}' + F_{N21},
\tag{2.16}
\end{equation*}

\textbf{Remark:} $F_{N20},F_{N20}',F_{N21}$ denote the force(or energy) density terms arising from the three-particle distribution $\phi_{N3}$ in equation \eqref{2.5}. The equation $2$ -particle couples to the $3$-particle density; the symbols $F_{N20},F_{N21}$ group these contributions. However, the following integration approach removed $\phi_{N3}$.

Multiplying \eqref{2.16} by $\dfrac{K}{2D}\Phi\big(|\xi_1|\big)$, integrating with respect to $u,\xi_1,v_1$, and integrating by parts to eliminate derivatives of $\phi_{N2}$ with respect to these variables, we obtain
\begin{equation*}\label{2.17}
\frac{\partial}{\partial t}
\left[
\frac{K}{2D}\int_{-\infty}^{\infty}\Phi\big(|\xi_1|\big)\,\phi_{N2}\,du\,d\xi_1\,dv_1
\right]
+
\frac{\partial}{\partial x^\beta}
\left[
\frac{K}{2D}\int_{-\infty}^{\infty}u^\beta \Phi\big(|\xi_1|\big)\,\phi_{N2}\,du\,d\xi_1\,dv_1
\right]
\tag{2.17}
\end{equation*}
\[
=
\frac{K}{D}
\int_{-\infty}^{\infty}
\Phi'\big(|\xi_1|\big)\,\frac{\xi_1^\beta}{|\xi_1|}
\left(\tfrac12 v_1^\beta\right)
\phi_{N2}\,du\,d\xi_1\,dv_1.
\]

\textbf{Remark:} The $\phi_{N3}$ terms vanish after integration since we can use Fubini and $\int_R\phi_{N3u}du=\phi_{N_3}|_{-\infty}^{+\infty}=0$ and $\phi_{Nl}$ vanishes at infinity ; only $\phi_{N2}$ remains in the averaged potential-energy balance.

Adding \eqref{2.11} and \eqref{2.17}, and using the definition of $e_N$ in \eqref{2.7} together with \eqref{2.14} and \eqref{2.13}, we obtain the energy equation
\begin{equation*}\label{2.18}
\begin{aligned}
&e_{N,t}(t;x)
+\frac{\partial}{\partial x^\beta}\big(\bar u_N^\beta e_N\big)
+\chi_{N2,\beta}^\beta(t;x)
\\&=
\frac{K}{2D}
\int_{-\infty}^{\infty}
\Phi'\big(|\xi_1|\big)\,\frac{\xi_1^\beta}{|\xi_1|}
\left(u^\beta+\tfrac12 v_1^\beta\right)
\cdot
\left[
\frac{\phi_{N2}(t;x,u;\xi_1,v_1)-\phi_{N2}(t;x-\epsilon\xi_1,u';\,\xi_1,v_1)}{\epsilon}
\right]\,du\,d\xi_1\,dv_1,
\end{aligned}
\tag{2.18}
\end{equation*}
where
\[
\chi_N^\beta(t;x)
=
\frac{1}{2}\int_{-\infty}^{\infty}(u^\beta-\bar u_N^\beta)
\left[
 |u|^2\,\phi_{N1}(t;x,u)
+\frac{K}{D}\int_{-\infty}^{\infty}\Phi\big(|\xi_1|\big)\,\phi_{N2}(t;x,u;\xi_1,v_1)\,d\xi_1\,dv_1
\right]du.
\]
Equation \eqref{2.18} is the formal total-energy balance. The left-hand side consists of the time derivative of the energy density, the transport of energy by the mean velocity, and the fluctuation flux $\chi_N^\beta$. The right-hand side is the symmetrized interaction work, written in a spatial difference-quotient form analogous to the force term in \eqref{2.15}.

\subsection{Summary of the formal hydrodynamic system}
In conclusion, equations \eqref{2.9}, 
\begin{equation*}
\rho_{N1,t}(t;x)+\frac{\partial}{\partial x^\alpha}\!\big[\rho_N(t;x)\,\bar{u}_N^\alpha(t;x)\big]=0.
\end{equation*} (space continuity)

\eqref{2.10} with $F^\alpha_{N1}$ given by \eqref{2.15} 
\begin{equation*}
    \frac{\partial}{\partial t}\big(\rho_N\bar u_N^\beta\big)
+\frac{\partial}{\partial x^\alpha}
\big(\rho_N\bar u_N^\alpha \bar u_N^\beta+\psi_{N1}^{\alpha\beta}\big)
=F_{N1}^\beta(t;x),
\end{equation*} (How momentum change in space and time)
and \eqref{2.18} (energy) forms a partial set of generalized hydro equations.

Here is what has been achieved: mass continuity is exact at the formal level; equations of momentum and energy are not closed because they still involve $\phi_{N2}$, which was obtained from BBGKY hierarchy. If we can obtain a concrete form of $\phi_2=\lim \phi_{N2}$ and take the limit of these balance laws properly, we will have the hydrodynamic equation. The next sections are devoted to Morrey’s attempt to construct such $\phi_{Nl}$.

\section{The Invariant Measure and an Ergodic-Type Assumption}\label{section3}

The formal balance laws of Section \ref{section2} were derived from invariant $N$-particle distributions $\phi_{Nl}$. Morrey's next idea is to replace such invariant distributions by a family of phase-space measures adapted to given macroscopic fields $(\rho(t;x),u(t;x),e(t;x))$. This is done by imposing cell-wise constraints on mass, momentum, and energy, thereby defining a constrained set of microscopic states. The remaining step is then to assume that averages over this constrained family correctly describe the macroscopic evolution. Section \ref{section3} explains this construction and isolates the corresponding ergodic-type hypothesis.

To this end, we start by considering the macroscopic state on small cells of phase space and use it to approximate the real $(\rho,u,e)$. Being specific, for each such cell $R$ of $x$-space and each $P_N=(x_1,..,x_N;u_1,...,u_N)$ in phase space $\mathbb{R}^{6N}$, we can define
\begin{align*}\label{3.0}
D_N^1(R)(P_N) &:= m_N \#\{x_i\in R\}, 
\\
D_N^{1+\alpha}(R)(P_N) &:= m_N \sum_{x_i\in R} u_i^\alpha(t), \qquad \alpha = 1,2,3, 
\\
D_N^5(R)(P_N) &:=\frac{m_N}{2} \sum_{x_i\in R}\Big[ |u_i|^2 + K\sum_{k=1}^N \Phi(\frac{|x_i - x_k|}{\epsilon}) \Big]
\end{align*}

The following assumption allows us to separate the whole spatial space $\mathbb{R}^3$ into small cells, and then, we can approximate the macroscopic states via these cells.

\begin{assumption}\label{assump:R_N}
For each $N$ large enough, we can select a finite family of disjoint spatial cells
\[
R_{N,1},\dots,R_{N,\nu_N},
\]
covering all but a negligible portion of the particles.  
Each cell $\rightarrow0$ but so slowly that as $N\to\infty$,  $\forall p\in\{1,2,...,\nu_N\}$
\[
\nu_N \to \infty,\; \operatorname{diam}(R_{N,p})\to 0,\;\text{and}\;\; \frac{\epsilon^3}{\mathrm{vol}(R_{N,p})} \to 0 \qquad,
\]
which means each cell can contain many particles. 

What's more, whenever $\{r_N\}$ is a sequence in $\{R_{N,p}\}$ chosen that $r_N$ close down to one point $x_0$, we assume, for each sequence of proper $\{P_N\}$,the following limits exist:  
\begin{equation*}
\lim_{N\rightarrow \infty}\frac{D^1_N(r_{N,p})(P_N)}{\operatorname{vol}(r_{N,p})}
=\rho(t;x_0),
\qquad
\lim_{N\rightarrow \infty}\frac{D_N^{1+\alpha}(r_N)(P_N)}{D_N^1(r_N)(P_N)}
=u^\alpha(t;x_0),
\end{equation*}
and
\[
\lim_{N\rightarrow \infty}\frac{D^5_N(r_p)(P_N)}{D^1_N(r_p)(P_N)}
=e(t;x_0).
\]

\end{assumption} 

For example, we can require $\{R_{N,p}\}$ as the lattice cells covering the ball $B(0,N^\frac{1}{9})$ and each lattice has diameter $1/N^\frac{1}{9}$. We need at most $O(N^{\frac{2}{3}})$ cells to achieve this.

\subsection{The Constraint Manifolds $M_N$ and The Induced Phase-Space Measure}

For each fixed time $t_1$, one assume the fluid has given macro-state $(\rho,u,e)$. We want our $N$-particle distribution meets such macro-state at least on each small cell. To be specific, $\forall R_{N,p}$, we require
\begin{align}  
    D_N^1(R_{N,p})(P_N) &= \int_{R_{N,p}}\rho(t_1,x)dx,\label{3.1a}\tag{3.1a}
\\
D_N^{1+\alpha}(R_{N,p})(P_N) &=  \int_{R_{N,p}}\bar{u}^\alpha(t_1,x)\,\rho(t_1,x)dx, \qquad \alpha = 1,2,3, \label{3.1b}\tag{3.1b}
\\
D_N^5(R_{N,p})(P_N) &= \int_{R_{N,p}}e(t_1,x)\rho(t_1,x)dx
\label{3.1c}\tag{3.1c}
\end{align}

If we see them as restrictions, we can define the \emph{constraint manifold}
\[
M_N(t_1):=M_N((\rho(t_1,\cdot),\bar{u}(t_1,\cdot),e(t_1,\cdot))
=
\Big\{ P_N\;\; \text{satisfying the restrictions \eqref{3.1a}, \eqref{3.1b} and \eqref{3.1c}}\Big\}.
\tag{3.2}\label{3.2}
\]
Thus $M_N(t_1)$ is a sub-manifold of $\mathbb{R}^{6N}$ consisting of all \textbf{micro-states} $P_N$ whose \textbf{macro-state}:
cell-averaged mass, momentum, and energy agree with the prescribed macroscopic fields at time~$t_1$.

To perform averages over $M_N(t_1)$, we project the Liouville measure of the full phase space $\mathbb{R}^{6N}$ onto $M_N(t_1)$ via the Jacobian arising from the constraint equations.

This step is totally geometric. The Liouville measure on the whole space is the canonical Lebesgue measure $dx_1du_1...dx_Ndu_N$. We have
$5\nu_N$ restrictions in total. And since $\nu_N << N$, such many restrictions give us a sub-manifold with big dimensions. This induce a measure, denoted as $\mu_{M_N}$ in $M_N(t_1)$, given by
\[\tag{3.3}\label{3.3}
\mu_{M_N}(S)=L\cdot\int_E |\frac{\partial (f_1,..,f_{5\nu_N})}{\partial(x_1,..,x_{5\nu_N})}|^{-1}dx_{5\nu_N+1}...dx_{6N}\,,
\]
where $f_j$ are the implicit restrictions and $E$ is the projection of $S$ on $(x_{5\nu_N}+1,...,x_{6N})$

The total measure is written
\[
\mu_{M_N}\big(M_N(t_1)\big)
=
\int_{M_N(t_1)} d\mu_{M_N}.
\]
Then, we can introduce the $l$--particle marginal density.

\textbf{definition}[Finite--$l$-Particle Marginal]
Define the $l$-density
\[
\pi_{N l}(t_1; x_1,u_1;\dots;x_l,u_l)
=\frac{\operatorname{Proj}_ld\mu_{M_N(t_1; x_1,u_1;\dots;x_l,u_l)}(x_1,u_1;\dots;x_l,u_l)}{\mu_{M_N}\big(M_N(t_1)\big)},
\tag{3.4}\label{3.4}
\]

However, as we discussed earlier, there is no evidence that such $\mu_{M_N}$ is invariant under the Newton flow \eqref{1.3}. The marginals defined from $\mu_{M_N}$
are therefore not identical to the invariant distributions used in Section \ref{section2}. In what follows, we treat them as approximate substitutes, under the heuristic assumption that the discrepancy becomes negligible in the hydrodynamic limit. This is one of the central unproved steps in the argument.

\subsection{An ergodic-type assumption.}
Then, Morrey introduced certain ergodic-type assumption to ensure the equations derived in Section \ref{section2} can have limits.

Now, for each $P\in M_N(t_1)$, we can treat $D^*_N(P)$ as discrete measures on $\mathbb{R}^3$ which has point mass on each $x_j$. And we hope the space average w.r.t $\mu_{M_N}$ of the changing rate of each $D^*_N(P)$ along its trajectory $P(t)$ will eventually converges to the time derivatives of $(\rho,u,e)$. This averaging process requires that the particle motion must experience every phase state $P$ in $M_N$ in a significant short time. We will try to explain it in Section\ref{Fourier}.  

\subsubsection{Fourier transform}
To formally describe this average, Morrey considered the Fourier Transform for $x\in\mathbb{R}^3$. We introduce Fourier transforms not because they are the final macroscopic variables, but because they provide a tractable way to differentiate the discrete observables in time. For example, it's difficult to tell what's the time derivative of $m_N\#\{x_j(t)\in R_N\}$ really is, even given the dynamics of $(x;u)$.

\begin{equation}
\begin{aligned}
F(\rho)(t;y)&=\int_{\mathbb{R}^3}e^{ixy}\rho(t,x)\,dx =\int_{\mathbb{R}^3\times\mathbb{R}^3}e^{ixy}\phi_1(t,x,u)\,dxdu\\
F(u^\alpha\cdot\rho)(t;y)& =\int_{\mathbb{R}^3\times\mathbb{R}^3}u^\alpha\,e^{ixy}\phi_1(t,x,u)\,dxdu\\
F(e)(t;y)&=\int_{\mathbb{R}^3}e^{ixy}e(t,x)\,dx\\& =\frac12\int_{\mathbb{R}^3}e^{ixy}[\int_{\mathbb{R}^3}|u|^2\,\phi_{1}(t;x,u)\,du
+K\int_{\mathbb{R}^3\times\mathbb{R}^3\times\mathbb{R}^3}\Phi\big(|\xi_1|\big)\,
\phi_{2}(t;x,u;\xi_1,v_1)\,du\,d\xi_1\,dv_1]\,dx\\
\end{aligned}
\label{3.5}\tag{3.5}
\end{equation}

and

\begin{equation}
\begin{aligned}
F(D^1_N)(t;y)&=m_N\sum_{j=1}^N e^{ix_j(t)\cdot y}\\
F(D^{1+\alpha}_N)(t;y)&=m_N\sum_{j=1}^N u_j^\alpha(t) e^{ix_j(t)\cdot y}\\ 
F(D^{5}_N)(t;y)&=\frac{1}{2}m_N\sum_{j=1}^N e^{ix_j(t)y}\Big(\|u_j(t)\|^2+K\sum_{k\neq j}^N\Phi(\|x_k(t)-x_j(t)\|/\epsilon)\Big).
\end{aligned}
\label{3.6}\tag{3.6}
\end{equation}

Then we take the time derivative of \eqref{3.6} and use the dynamics \eqref{1.3}:
\begin{equation}
\begin{aligned}
 \partial_t[F(D^1_N)(t;y)]&=m_N\sum_{n=1}^N iy\,\dot{x}(t)e^{ix(t)\cdot y}=m_N\sum_{n=1}^N iy\,u(t)e^{ix(t)\cdot y}\\[2pt]
\partial_t[F(D^{1+\alpha}_N)(t;y)] &=\sum_{\beta=1}^3i\,y^\beta\, m_N[\sum_{j=1}^Nu^\alpha_ju_j^\beta e^{iy x_j}+\sum_{j=1}^{N} e^{iyx_j}
 \Big( \sum_{k\neq j}^{N} E\big(i\,y(x_k-x_j)\big)\, v_{Njk}^{\alpha\beta}\Big)],\\[2pt]
\partial_t[F(D^{5}_N)(t;y)] &=\sum_{\alpha,\beta}^3 \Big[i\,y^\beta\, m_N\sum_{j=1}^{N}u_j^\beta e^{iyx_j}\,\frac{1}{2}[\|u_j(t)\|^2+K\sum_{k\neq j}^{N}\Phi(\|x_k(t)-x_j(t)\|/\epsilon)]
 \\& +\ i\,y^\alpha\, m_N\sum_{j=1}^{N}u_j^\alpha e^{iyx_j}
 ( \sum_{k\neq j}^{N} E\big(i\,y_\beta(x_k-x_j)^\beta\big)\, v_{Njk}^{\beta\alpha})\Big],
\end{aligned}
\label{3.7}\tag{3.7}
\end{equation}

where $E(z):=(e^z-1)/z,\,E(0)=1$ and
\[
v_{Njk}^{\alpha\beta}
:=\frac{K}{2}\,(\|x_k-x_j\|/\epsilon)\,\Phi'(\|x_k-x_j\|/\epsilon)\,
\frac{(x_k^\alpha-x_j^\alpha)(x_k^\beta-x_j^\beta)}{\|x_k-x_j\|^2}.
\]

On certain $M_N(t_1)$, the Ergodic theorem says that the time average of the any given good $f$ along trajectory $P_N(t)$ over a sufficient long time($\rightarrow0$ as $N\rightarrow\infty$ in our case) will equal to the space average of over $M_N$ under the measure $\mu_{M_N}$ which is "invariant" along $P_N(t)$. Therefore, we shall assume, for example 
\begin{align*}
    \partial_t[F(\rho)(t;y)]&=\lim_{N\rightarrow\infty}\int_{M_N(t)}\partial_t[F(D^1_N)(t;y)]\,\frac{d\mu_{M_N}(P_N)}{\mu_{M_N}(M_N)}\\
    &=\lim_{N\rightarrow\infty}\int_{M_N(t)}m_N\sum_{n=1}^N iy\,ue^{ix\cdot y}\,\frac{d\mu_{M_N}(P)}{\mu_{M_N}(M_N)}\,.
\end{align*}
This discussion tells us what ergodic assumption we should expect.
 
\begin{assumption}[Ergodic type hypothesis]\label{assump:ergodic}
    For each $(t,y)$, the time derivative of the limiting function $(\rho,\bar{u},e)$ satisfies
\begin{align}\label{3.8}
 \partial_t[F(\rho)(t;y)]&=\lim_{N\rightarrow\infty}\int_{M_N(t)}\partial_t[F(D^1_N)(t;y)]\,\frac{d\mu_{M_N}(P_N)}{\mu_{M_N}(M_N)}
\tag{3.8a}\\[4pt]
\partial_t[F(u^\alpha\cdot\rho)(t;y)]&=\lim_{N\rightarrow\infty}\int_{M_N(t)}\partial_t[F(D^{1+\alpha}_N)(t;y)]\,\frac{d\mu_{M_N}(P_N)}{\mu_{M_N}(M_N)}
\tag{3.8b}\\[4pt]
\partial_t[F(e)(t;y)]&=\lim_{N\rightarrow\infty}\int_{M_N(t)}\partial_t[F(D^5_N)(t;y)]\,\frac{d\mu_{M_N}(P_N)}{\mu_{M_N}(M_N)}\tag{3.8c} 
\end{align}
\end{assumption}

It's clear that RHS is the space average. LHS is the time average when the traversal time $\rightarrow0$ as $N\rightarrow\infty$.

\subsubsection{Explanation of the Ergodic-Type Assumption \ref{assump:ergodic}}
Assumption \ref{assump:ergodic} asserts that, after averaging over the constrained phase-space family and passing to the hydrodynamic limit, the macroscopic time derivatives are determined by the same observables that appear in the microscopic Fourier dynamics. This is indeed much stronger than the classical ergodic-type assumption needed in similar projects.\cite{Khinchin1949}

Fubini's theorem tells us, if we assume the regularity of the integrant,
\begin{align*}
    \int_{M_N(t)}\partial_t[F(D^*_N)(t;y)]\,\frac{d\mu_{M_N}(P_N)}{\mu_{M_N}(M_N)}&=\partial_tF(\int_{M_N(t)}(D^*_N)(t;x)\,\frac{d\mu_{M_N}(P_N)}{\mu_{M_N}(M_N)})(y)\;.
\end{align*}
If one could identify the averaged discrete distributions on the right with the macroscopic fields $(\rho_N,\rho_N\bar{u}_N,e_N)$, then Assumption 3 would amount to a time-average/phase-average replacement principle for these observables.

Then, taking the Fourier inverse in equation (3.8), we will have $\partial_t\rho_N(t,x)\rightarrow\partial_t\rho(t,x)$, $\partial_t\rho_N\bar{u}_N(t,x)\rightarrow\partial_t\rho u(t,x)$ and $\partial_te_N(t,x)\rightarrow\partial_te(t,x)$ weakly. Formally, this would permit one to pass to the weak limit for equations \eqref{2.9},\eqref{2.10} and \eqref{2.18} and obtain fluid equations.

In conclusion, to derive the Euler equation, we have to assume two things. First, we can replace $(\rho_N,\rho_N\bar{u}_N,e_N)(t,x)$ in \eqref{2.9},\eqref{2.10} and \eqref{2.18} with $\int_{M_N(t)}(D^*_N)(t;x)\,\frac{d\mu_{M_N}(P_N)}{\mu_{M_N}(M_N)}$. The difficulty is that the $\mu_{M_N}$ is only invariant along the trajectory of the particle motion projected on $M_N$, while the derivation in Section \ref{section2} requires $d\mu_{M_N}(P_N)$ is invariant along the flow \eqref{1.3} itself. Second, equations  \eqref{2.9},\eqref{2.10} and \eqref{2.18} have limits, which is indicated by the ergodic assumption \ref{assump:ergodic}.

Some efforts are made to try to justify these assumptions in Section \ref{Fourier}. However, some gaps remain to be filled here.

\subsection{Towards the Hydrodynamic Limit}

Assuming that the constrained marginals $\pi_{Nl}$ admit a sufficiently strong Gibbs-type limit, and assuming that Assumption \ref{assump:ergodic} identifies the limiting time derivatives correctly, one ibhs formally led to the hydrodynamic equations obtained from the balance laws of Section \ref{section2}. At this point no rigorous proof of this convergence is available in the draft; the purpose of the subsection is only to indicate the intended limiting mechanism. Moreover, in carrying out this program one finds that the finite--$N$ marginals $\pi_{N l}$ converge to functions of the Gibbs form \eqref{6.1}.

\textbf{Remark :} This assumption $\phi_{Nl}$ converge uniformly to $\phi_l$ is \textbf{not} rigorously proved, see Section \ref{Section:limit}.

\section{Partial justification of the ergodic-type hypothesis}\label{Fourier}

In this section, we will try to partially justify the ergodic assumption \ref{assump:ergodic} and explain why it is plausible to treat $\pi_{NN}$ defined by the equation \eqref{3.3} as local equilibrium measures even though they are not exactly invariant. To this end, let us return to the particle motion and use the bounds (Theorems~\ref{thm:LNbound}--\ref{thm:A.2} in the Appendix), which in turn control moments of the particle system and enable compactness arguments for certain transforms.

\subsection{Discussion of the ergodic assumption \ref{assump:ergodic}}
This is the point where Morrey's own argument shifts from exact balance laws to compactness and averaging heuristics.

Recall that in Section \ref{section3}, we introduced Fourier–Stieltjes transforms of the \(x\)-distributions for discrete measures $D^*_N$: \eqref{3.6} and their time derivatives \eqref{3.7}. For the readers' convenience, let us copy \eqref{3.7} here.
\begin{equation*}
\begin{aligned}
 \partial_t[F(D^1_N)(t;y)]&=m_N\sum_{n=1}^N iy\,\dot{x}(t)e^{ix(t)\cdot y}=m_N\sum_{n=1}^N iy\,u(t)e^{ix(t)\cdot y}\\[2pt]
\partial_t[F(D^{1+\alpha}_N)(t;y)] &=\sum_{\beta=1}^3i\,y^\beta\, m_N[\sum_{j=1}^Nu^\alpha_ju_j^\beta e^{iy x_j}+\sum_{j=1}^{N} e^{iyx_j}
 \Big( \sum_{k\neq j}^{N} E\big(i\,y(x_k-x_j)\big)\, v_{Njk}^{\alpha\beta}\Big)],\\[2pt]
\partial_t[F(D^{5}_N)(t;y)] &=\sum_{\alpha,\beta}^3 \Big[i\,y^\beta\, m_N\sum_{j=1}^{N}u_j^\beta e^{iyx_j}\,\frac{1}{2}[\|u_j(t)\|^2+K\sum_{k=1}^{N}\Phi(\|x_k(t)-x_j(t)\|/\epsilon)]
 \\& +\ i\,y^\alpha\, m_N\sum_{j=1}^{N}u_j^\alpha e^{iyx_j}
 ( \sum_{k\neq j}^{N} E\big(i\,y_\beta(x_k-x_j)^\beta\big)\, v_{Njk}^{\beta\alpha})\Big],
\end{aligned}
\end{equation*}
where $E(z):=(e^z-1)/z,\,E(0)=1$ and
\[
v_{Njk}^{\alpha\beta}
=\frac{K}{2}\,(\|x_k-x_j\|/\epsilon)\,\Phi'(\|x_k-x_j\|/\epsilon)\,
\frac{(x_k^\alpha-x_j^\alpha)(x_k^\beta-x_j^\beta)}{\|x_k-x_j\|^2}.
\]

Then, if we can prove that \eqref{3.7} are uniformly bounded, by Ascoli–Arzelà, a uniformly convergent subsequence of \eqref{3.6} can be drawn. To this end, we will need certain estimations about the summation of $\Phi$, which will be presented in Appendix \ref{sectionAbounds}. This is the analytical hinge for the compactness theorems that follow. However, the boundedness of $\partial_t[F(D^{5}_N)(t;y)]$ remains to be a \emph{gap}. 

\begin{theorem}[4.1]\label{thm:5.3}
Suppose for some \(t_0\) we have a sequence of particle systems with \(N\to\infty\) obeying the scaling law \eqref{1.1} and their total energy $E_N$  and $C_N$ in \eqref{A.8} are bounded in $N$.
Then there exists an infinite subsequence (not relabeled) for which
\(F(D^\gamma_N)(t;y) \to \psi^\gamma\) uniformly on every bounded subset of \((t,y)\)-space
for \(\gamma=1,2,3,4\). Each limit \(\psi^\gamma\) satisfies a uniform Lipschitz condition on bounded sets.
Moreover \(\psi^\gamma\) is continuously differentiable in both \(t\) and \(y^\alpha\) and satisfies the first two identities in \eqref{3.7}.
\end{theorem}

\textbf{Remark:} The boundedness of $E_N$ is obtained by restriction \eqref{5.10} as long as we assume that the total macroscopic energy $\int_\mathbb{R}^3e(t,x)dx$ is finite. Because we require the particle systems to approximate cell-wise the given $e(t,x)$.

\begin{proof}
By \eqref{A.8} and \eqref{A.6}, both
\(m_ N\sum \|x_i\|^2\) and \(m_ N\sum \|u_i\|^2\) are uniformly bounded on finite time intervals. Hence the sums defining $F(D^\gamma_N)(t;y)$ are uniformly bounded. 

For the second equation in \eqref{3.7}, 
\[
\partial_t[F(D^{1+\alpha}_N)(t;y)] =\sum_{\beta=1}^3i\,y^\beta\, m_N[\sum_{j=1}^Nu^\alpha_ju_j^\beta e^{iy x_j}+\sum_{j=1}^{N} e^{iyx_j}
 \Big( \sum_{k\neq j}^{N} E\big(i\,y(x_k-x_j)\big)\, v_{Njk}^{\alpha\beta}\Big)],
\]
we can see that the factor \(E(iz)\) satisfies \(|E(iz)|\le 1\) for real \(z\). Hence we only have to estimate $m_ N\sum_{i\neq k}v_{Njk}^{\alpha\beta}$. By Theorem \ref{thm:LNbound}, we know $\sum (\|x_k-x_i\|/\epsilon)\,\Phi'(\|x_k-x_i\|/\epsilon)$ is controlled by $\sum \Phi$ and $L$. So $m_N\sum \Phi$ is the potential energy.  Noticing that the scaling law \eqref{1.1} implies that \(M_N\), \(E_N\) and \(C_N\) (the quantity in \eqref{A.7}) remain bounded. Thus, time derivatives of $F(D^\gamma_N)(t;y)$ are uniformly bounded in \(L^\infty\) on bounded \((t,y)\)-domains.

Since $F(D^\gamma_N)(t;y)$ are equi-continuous and uniformly bounded; by Ascoli–Arzelà, there is a uniformly convergent subsequence on compact sets, yielding the statement.
\end{proof}

This convergence theorem gives us the intuition of assumption \ref{assump:ergodic}. However, we don't know if the limits $\psi^\gamma$ are indeed the Fourier transforms of $(\rho,u,e)$ given by \eqref{3.5}. What is more, Nothing further is said here about the function $F(D^5_N)(t;y)$, since its derivative involves third moment with respect to $u$ and certain mixed cross--moments which have not been proved to remain bounded in time. However, at least the factor $\epsilon^{-1}$ does \emph{not} appear in $\partial_t[F(D^{5}_N)(t;y)]$, which means its analytical form is still similar to that of the $\partial_t[F(D_{N}^{1+\alpha})(t;y)]$. 

In the remainder of the discussion, we make the additional hypothesis that an analogous compactness conclusion in Theorem \ref{thm:5.3} also holds for the fifth observable.

\begin{conjecture}\label{conj1}
    Theorem \ref{thm:5.3} also holds for $\gamma=5$.
\end{conjecture}

\subsection{ $\pi_{NN}$ in \eqref{3.3} and invariant measure }\label{sec:ergodic}
The discussion in this subsection is heuristic. We will try to argue that the $\pi_{NN}$ behave like the invariant measure along the Newton flow \eqref{1.3}. Recall that $\pi_{NN}$ is invariant along the projected trajectory. We want to show that the projection on $M_N$ won't change the flow too much. 

Since the functions $\partial_t[F(D_{N}^{\gamma})(t;y)]$ and their time derivatives are uniformly bounded independent of $N$, the change of $(\rho_N,\rho_N\bar{u}_N,e_N)$ is small under a small time $t_\epsilon$. This suggests that the trajectories of the $N$--particle system remain close to $M_N$. In other words, $\pi_{NN}$, which is invariant along the projected trajectory, only have $\epsilon$-change along the real flow. 

On the other hand, considering the motion law \eqref{1.3}, the equation $\dot u_i^\alpha$ contains the factor $\epsilon^{-1}$, so the microscopic accelerations are significantly large as $N\to\infty$. Particles exchange momentum rapidly and experience all $P\in M_N$ in a short time $t_\epsilon$.

Thus the trajectory of the phase point $P(t)$ decomposes into a \emph{rapid} motion along the manifold $M_N$ and a \emph{slow} drift orthogonal to $M_N(t)$ as the macroscopic fields change. Therefore, one can expect that the induced measure $\pi_{NN}$ from \eqref{3.3} is $\epsilon$-close to the invariant $\pi_{NN}$ introduced in Section \ref{section2}. Hence, it's convincing to have the following conjecture:

\begin{conjecture}\label{conj:pi}
    For the marginals $\pi_{Nl}$ and $\phi_{Nl}$ induced by the constrained measure $\mu_{M_N}$ \eqref{3.3}, the equation \eqref{2.5} and the balance laws \eqref{2.9}, \eqref{2.10}, and \eqref{2.18} will hold up to certain error $r_N$ which will vanish after taking the weak hydrodynamic limit. For example, we expect the following to hold weakly :\[\frac{\partial\rho_N(t;x)}{\partial t}+\frac{\partial}{\partial x^\alpha}\!\big[\rho_N(t;x)\,\bar{u}_N^\alpha(t;x)\big]=r_N\,.\]
\end{conjecture}

In addition, such invariance-like conjecture exactly corresponds to $(i)$ of the definition \ref{def:localequi}. Proving this is one of the main gap in Morrey's paper and requires the ergodic assumption.

\section{Derivation of $\phi_{Nl}$ via projection \eqref{3.3}}
\label{Section:phiNL}
The following Sections \ref{Section:phiNL} and \ref{Section:limit} are the most important while having most gaps. 

In this section, we derive an explicit finite-$N$ formula for $\phi_{Nl}$ starting from the constrained measure introduced in \eqref{3.3}. The goal is to understand what the $l$-particle marginals of the constrained phase-space measure look like, and how they lead to the local form later used in the limiting argument. 

The computation has three steps. First, we decompose $M_N$ into symmetrically placed pieces according to which particles lie in which cells. Second, we write the marginal density as a combinatorial factor times a conditional Liouville-volume ratio. Third, we compute this ratio by solving the cellwise momentum and energy constraints and integrating out the constrained variables.

\subsection{Decomposition of $M_N$ and the marginal $\pi_{Nl}$}

Recall that $M_N(t_1)$ is defined by the cellwise constraints
\[
D_N^\gamma(t_1;R_\nu)=C.
\]
The difficulty is that these constraints fix only how many particles lie in each cell, not which labels they carry. To make the induced measure more explicit, we therefore decompose $M_N$ into labeled pieces.

Let
\[
c_p:=\#\{x_j\in R_p\}.
\]
Then $M_N$ is the disjoint union of
\begin{equation}
\frac{N!}{c_1!\cdots c_P!},
\qquad
c_1+\cdots+c_P=N,
\tag{5.1}\label{5.1}
\end{equation}
symmetrically placed manifolds $M_{J_N}$, where $J_N=(S_1,\dots,S_P)$ is a sequence of disjoint subsets of $\{1,\dots,N\}$ such that $|S_p|=c_p$ for each $p$. On each $M_{J_N}$, the cell membership of every particle is fixed.

For simplicity, we now regard $N$ as fixed and drop the subscript $N$ from $c_p$, $P_N$, $R_{Np}$, and similar quantities. We also assume that the smallest $c_p\to\infty$ as $N\to\infty$.

Now let
\[
(x_{01},u_{01};\dots;x_{0l},u_{0l})
\]
be prescribed. If some $x_{0i}$ lies outside the chosen cells, then $\pi_{Nl}=0$. So we assume that the prescribed points are distributed among the cells according to
\begin{equation}
x_{0,L_{i-1}+k}\in R_{p_i},
\qquad
k=1,\dots,l_i,
\qquad
L_0=0,
\qquad
L_i=\sum_{m=1}^i l_m,
\qquad
L_r=l.
\tag{5.2}\label{5.2}
\end{equation}
Here $l_i$ counts how many of the prescribed points lie in the cell $R_{p_i}$.

For $M_{J_N}(x_{01},u_{01};\dots;x_{0l},u_{0l})$ to be nonempty, the label assignment $J_N$ must place the first $l$ prescribed particles into the correct cells. The number of such admissible assignments is
\begin{equation}
\frac{(N-l)!}{(c_1-l_1)!\cdots(c_r-l_r)!\,c_{r+1}!\cdots c_{\nu_N}!}.
\tag{5.3}\label{5.3}
\end{equation}
This is a purely combinatorial factor: it counts how many labelings remain possible after fixing the cells of the first $l$ particles.

We now compute $\pi_{Nl}(x_{01},u_{01};\dots;x_{0l},u_{0l})$ from \eqref{3.4}, \eqref{5.1}, and \eqref{5.3}. Since $M_N$ is the disjoint union of the manifolds $M_{J_N}$ with equal weight, the marginal splits into a combinatorial part and a geometric part:
\begin{equation}
\pi_{Nl}(x_{01},u_{01};\dots;x_{0l},u_{0l})
=
\frac{c_1!\cdots c_r!\,(N-l)!}{(c_1-l_1)!\cdots(c_r-l_r)!\,c_{r+1}!\cdots c_{\nu_N}!\,N!}
\;\kappa_{Nl},
\tag{5.4}\label{5.4}
\end{equation}
where
\[
\kappa_{Nl}
:=
\frac{
\operatorname{Proj}_l\mu_{M_{J_N}}
\!\left[M_{J_N}(x_{01},u_{01};\dots;x_{0l},u_{0l})\right]
}{
\mu_{M_{J_N}}(M_{J_N})
}.
\]

Formula \eqref{5.4} has two parts. The first is the combinatorial factor coming from the number of admissible labelings. The second, $\kappa_{Nl}$, is a conditional Liouville-volume ratio on a fixed labeled component $M_{J_N}$.

Here $\mu_{M_{J_N}}$ denotes the Liouville measure induced on the manifold $M_{J_N}$, and
\[
M_{J_N}(x_{01},u_{01};\dots;x_{0l},u_{0l})
\]
is the slice of $M_{J_N}$ obtained by fixing the first $l$ phase points. Thus $\kappa_{Nl}$ measures the relative phase-space volume of this constrained slice.

\subsection{Cell variables and the Jacobian}

We now compute $\kappa_{Nl}$ more explicitly with given macro-state $(\rho(t,x),u(t,x),e(t,x))$. For each cell $R_p$ at the fixed time $t_1$, define
\begin{equation}
\rho_p
:=
\frac{D^1(R_p)}{\operatorname{vol}(R_p)}
=
\frac{\int_{R_p}\rho(t_1,x)\,dx}{\operatorname{vol}(R_p)},
\qquad
\bar u_p
:=
\frac{D^{1+\alpha}(R_p)}{D^1(R_p)}
=
\frac{\int_{R_p}\bar u(t_1,x)\rho(t_1,x)\,dx}{\int_{R_p}\rho(t_1,x)\,dx},
\tag{5.6}\label{5.6}
\end{equation}
and
\[
\varepsilon_p+\frac12|\bar u_p|^2
:=
\frac{D^5(R_p)}{D^1(R_p)}
=
\frac{\int_{R_p}e(t_1,x)\,dx}{\int_{R_p}\rho(t_1,x)\,dx}.
\]
These are the cell-wise density, mean velocity, and total energy.

To simplify notation, write
\begin{equation}
x_{pj}=x_{L_{p-1}+j},
\qquad
x_{0pj}=x_{0,L_{p-1}+j},
\qquad
u_{pj}=u_{L_{p-1}+j},
\qquad
u_{0pj}=u_{0,L_{p-1}+j},
\tag{5.8}\label{5.8}
\end{equation}
for the particles associated with cell $R_p$.

Next, introduce centered velocities
\begin{equation}
v_{pj}=u_{pj}-\bar u_p,
\qquad
v_{0pj}=u_{0pj}-\bar u_p.
\tag{5.9}\label{5.9}
\end{equation}
This change of variables separates the prescribed bulk motion in each cell from the internal fluctuations. In these variables, the momentum constraints become linear, and the energy constraint is written in terms of the internal energy.

The manifold $M_{J_N}$ is then described by the restrictions
\begin{equation}
x_{pj}\in R_p,
\qquad
\sum_{j=1}^{c_p}v_{pj}^i=0,\quad i=1,2,3,
\qquad
\sum_{j=1}^{c_p}\bigl(|v_{pj}|^2+W_p(x)\bigr)=2c_p\varepsilon_p,
\tag{5.10}\label{5.10}
\end{equation}
where
\[
W_p(x)
=
K\sum_{i,j=1}^{c_p}\Phi(|x_{pi}-x_{pj}|/\epsilon)
+
K\sum_{q\neq p}\sum_{i=1}^{c_p}\sum_{k=1}^{c_q}
\Phi(|x_{pi}-x_{qk}|/\epsilon).
\]
The quantity $W_p(x)$ is the total interaction energy seen by the particles in the cell $R_p$, including both the interactions inside $R_p$ and those with particles in the other cells.

Each cell contributes three momentum constraints and one energy constraint. We therefore solve for four variables in each cell, for example
\[
v_{p,c_p}^1,\quad v_{p,c_p}^2,\quad v_{p,c_p}^3,\quad v_{p,c_p-1}^1,
\]
and use the remaining variables as free coordinates. The corresponding Jacobian in \eqref{3.3} becomes
\begin{equation}
\mathcal J=\pm\prod_{p=1}^P\mathcal J_p,
\qquad
|\mathcal J_p|=2^{3/2}K_p^{1/2},
\tag{5.11}\label{5.11}
\end{equation}
where
\begin{equation}
A_{pi}=\sum_{k=1}^i v_{pk},
\qquad
K_p
=
2c_p\varepsilon_p
-
W_p(x)
-
\sum_{j=1}^{c_p-2}|v_{pj}|^2
-\frac12|A_{p,c_p-2}|^2
-
2\sum_{\alpha=2}^3
\left(v^\alpha_{p,c_p-1}+\frac12A^\alpha_{p,c_p-2}\right)^2.
\tag{5.12}\label{5.12}
\end{equation}
The sign $\pm$ comes from solving the quadratic energy constraint. The quantity $K_p$ is the remaining energy budget in the cell after the constrained variables have been eliminated, so the condition $K_p>0$ describes the region where the constraints are compatible.

\subsection{The conditional volume ratio}

For fixed $(x_{01},u_{01};\dots;x_{0l},u_{0l})$, the free variables range over the product of the domains $G_p$ determined by $K_p>0$. Writing $\hat x$ for the remaining free position variables, we obtain
\begin{equation}
\operatorname{Proj}_l\mu_{M_{J_N}}(x_{01},u_{01};\dots;x_{0l},u_{0l})
=
\int_{G_1\times\cdots\times G_P}
\prod_{p=1}^r g_{p,l_p}(x_{0p},v_{0p};\hat x)\,d\hat x,
\tag{5.13}\label{5.13}
\end{equation}
where
\[
g_{p,l_p}(x_{0p},v_{0p};\hat x)
=
L_p
\int_{K_p>0}
K_p^{-1/2}\,
dv_{p,l_p+1}\cdots dv_{p,c_p-2}\,dv_{p,c_p-1}^2\,dv_{p,c_p-1}^3.
\]
Here $g_{p,l_p}$ is the contribution of the cell $R_p$ to the restricted Liouville volume, and $L_p$ is a constant depending only on the dimension of the constraint manifold.

To evaluate $g_{p,l_p}$, we first integrate with respect to $v_{p,c_p-1}^2$ and $v_{p,c_p-1}^3$. When $l_p=c_p-2$, this gives
\begin{equation}
g_{p,c_p-2}(x_0,v_0;\hat x)
=
L
\left\{
\mathrm{pos}
\left[
2c_p\varepsilon_p
-
W(x)
-
\sum_{i=1}^{c_p-2}|v_i|^2
-
\frac12|A_{c_p-2}|^2
\right]
\right\}^{1/2},
\tag{5.14}\label{5.14}
\end{equation}
where
\[
\mathrm{pos}\,y := y\,\mathbf 1_{\{y\ge 0\}}.
\]

Next, for $1\le q\le c_p-2$, we use the identity
\[
\sum_{j=1}^q|v_j|^2+\frac{1}{c_p-q}|A_q|^2
=
\sum_{j=1}^{q-1}|v_j|^2+\frac{1}{c_p-q+1}|A_{q-1}|^2
+
\frac{c_p-q+1}{c_p-q}
\left|v_q+\frac{1}{c_p-q+1}A_{q-1}\right|^2.
\]
Iterating this gives
\begin{align}
\sum_{i=1}^{c_p-2}|v_i|^2+\frac12|A_{c_p-2}|^2
&=
\frac32\left|v_{c_p-2}+\frac13A_{c_p-3}\right|^2
+
\frac43\left|v_{c_p-3}+\frac14A_{c_p-4}\right|^2
+\cdots
\notag\\
&\quad
+
\frac{c_p-l}{c_p-l-1}\left|v_{l+1}+\frac{1}{c_p-l}A_l\right|^2
+
\sum_{i=1}^l|v_i|^2
+
\frac{1}{c_p-l}|A_l|^2.
\tag{5.15}\label{5.15}
\end{align}
This identity allows us to integrate successively with respect to
\[
v_{c_p-2},v_{c_p-3},\dots,v_{l_p+1}
\]
over all space. We therefore obtain
\begin{align}
g_{p,l}(x_{0l},v_{0l};\hat x)
&=
L\cdot K^*_{c_p,l_p}
\left\{
\mathrm{pos}
\left[
2c_p\varepsilon_p
-
W(x)
-
\sum_{i=1}^{l_p}|v_i|^2
-
(c_p-l_p)^{-1}|A_l|^2
\right]
\right\}^{(3c_p-3l_p-5)/2}
\notag\\
&=
L\cdot K_{c_p,l_p}
\left(\frac{B}{\pi}\right)^{3l/2}
\left\{
\mathrm{pos}
\left[
1-\frac{2B}{3c_p}W(x)-\frac{2B}{3c_p}U_{c_p,l}
\right]
\right\}^{(3c_p-3l_p-5)/2},
\tag{5.16}\label{5.16}
\end{align}
where
\begin{align*}
B&=\frac{3}{4\varepsilon_p},\\
K_{c,l}
&=
\frac{
2^{3/2}\Gamma(3/2)\,\pi^{3(c-2)/2}\,(2c\varepsilon_p)^{(3c-5)/2}
}{
(c-l)^{3/2}(3c/2)^{3/2}\Gamma\!\bigl(3(c-1-l)/2\bigr)
},\\
U_{c,l}
&=
\sum_{i=1}^l|v_i|^2+(c-l)^{-1}|A_l|^2.
\end{align*}

Formula \eqref{5.16} is the key finite-$N$ output of this section. It shows that, after eliminating the constraints, the conditional marginal is governed by a cutoff power of the remaining interaction energy and fluctuation energy. This is the finite-$N$ precursor of the Gibbs-type local equilibrium form that appears later in the limiting argument.

\subsection{Restriction to one cell and the resulting formula}

To carry out the successive integration in \eqref{5.13} explicitly, we now make a simplifying assumption: all prescribed particles $x_{01},\dots,x_{0l}$
lie well inside the same typical cell $R_1$. This is a real restriction, and it is one of the gaps in the argument.

\begin{conjecture}\label{conj:cell}
For $l$ being small, the functions $\phi_{Nl}$ derived under the assumption that
\[
x_{01},\dots,x_{0l}\in R_1
\]
converge to the same limiting function $\phi_l$ as in the general case.
\end{conjecture}
Since each cell can contain many particles and what we care is actually $l$ small , we can imagine these particles either in the same cell or independent in the following derivation.

Under this one-cell restriction, the interaction terms simplify. If the cross-cell terms
\[
\Phi(|x_{pj}-x_{qk}|/\epsilon)
\]
were absent from $W_p(x)$, then each $g_{p,l}$ would depend only on the variables from its own cell, and \eqref{5.13} would factor completely. In the present simplified setting, only the contribution from $p=1$ remains nontrivial, while the other cells contribute through normalization. Using \eqref{5.16}, we obtain
\begin{equation}
\kappa_{Nl}(x_{01},u_{01};\dots;x_{0l},u_{0l})
=
\hat K_{N,l}
\left(\frac{B}{\pi}\right)^{3l/2}
\frac{
\displaystyle
\int_{R_1^{c_1-l}}
\left[
\mathrm{pos}
\left(
1-\frac{2B_1}{3c_1}W_1(x)-\frac{2B_1}{3c_1}U_{c_1,l}
\right)
\right]^{(3c_1-3l-5)/2}
dx_{l+1}\cdots dx_{c_1}
}{
\displaystyle
\int_{R_1^{c_1}}
\left[
\mathrm{pos}
\left(
1-\frac{2B_1}{3c_1}W_1(x)
\right)
\right]^{(3c_1-5)/2}
dx_1\cdots dx_{c_1}
}.
\tag{5.17}\label{5.17}
\end{equation}
Here
\[
\hat K_{c_1,l}=\frac{K_{c_1,l}}{K_{c_1,0}},
\qquad
\lim_{c_1\to\infty}\hat K_{c_1,l}=1.
\]
Since $R_1$ is arbitrary and $c_1\to\infty$, we now drop the subscript $1$, replace $c_1$ by $N$, and introduce
\begin{equation}
n=\frac{3N}{2},
\qquad
k_0=\frac52,
\qquad
k_1=\frac{3l+5}{2}.
\tag{5.18}\label{5.18}
\end{equation}
This simplification needs more discussion, see the later remark.

For $N\ge l+1$ and real $k$, define
\begin{align}
C^{\,n-k}_{Nkl}(\rho,b)
&:=
r^{-N+l}
\int_{R^{N-l}}
\left[
\mathrm{pos}
\left(
1-\frac{b}{n}W_{Nl}^{**}
\right)
\right]^{n-k}
dx_{l+1}\cdots dx_N,
\tag{5.19a}\label{5.19a}\\
W_{Nl}^{**}
&:=
K\sum_{i,k=l+1;i\neq k}^N
\Phi(|x_i-x_k|/\epsilon),
\notag\\
W_{Nl}^{*}
&:=
2K\sum_{j=1}^l\sum_{k=l+1}^N
\Phi(|x_j-x_k|/\epsilon),
\notag\\
r&:=\mu(R),
\qquad
\rho:=ND\epsilon^3/r.
\notag
\end{align}
The dependence on $\rho$ of $C_{Nkl}^{n-k}$ in \eqref{5.19a} comes from considering the change of variable $\xi=x/\epsilon$. Indeed, $C_{Nkl}^{n-k}$ also should depend on the shape of $R$. However, by \emph{van Hove theorem} \cite{Ruelle1969}, the dependence vanishes after taking the hydrodynamic limit. We don't discuss it in detail in this paper. For simplicity, we can assume the shape of $R$ as fixed.

Similarly,
\begin{equation}
D^{\,n-k}_{Nkl}(x_1,\dots,x_l;\rho,b)
:=
r^{-N+l}
\int_{R^{N-l}}
\left[
\mathrm{pos}
\left(
1-\frac{b}{n}W_{Nl}^{**}-\frac{b}{n}W_{Nl}^{*}
\right)
\right]^{n-k}
dx_{l+1}\cdots dx_N.
\tag{5.19b}\label{5.19b}
\end{equation}

With this notation, \eqref{5.17} becomes
\begin{equation}
\kappa_{Nl}
=
\hat K_{N,l}
\left(\frac{B}{\pi}\right)^{3l/2}
\left[
\mathrm{pos}
\left(
1-\frac{B}{n}W_{Nl}-\frac{B}{n}U_{Nl}
\right)
\right]^{n-k_1}
\left[
\frac{C_{Nk_0l}(\rho,B')}
     {C_{Nk_0l}(\rho,B)}
\right]^{n-k_0}
\gamma_{Nk_1l}(x_1,\dots,x_l;\rho,B'),
\tag{5.20}\label{5.20}
\end{equation}
where
\begin{align*}
\gamma_{Nkl}
&:=
\frac{D^{\,n-k}_{Nkl}}{C^{\,n-k}_{Nk_0l}}\;,\;\;\;
B':=
\frac{B}{1-\frac{B}{n}W_{Nl}-\frac{B}{n}U_{Nl}}\,,\\
W_{Nl}
&:=
K\sum_{i,j=1;\,i\neq j}^l\Phi(|x_i-x_j|/\epsilon).
\end{align*}

\textbf{Remark :}
Three points deserve emphasis in passing 
from~\eqref{5.17} to~\eqref{5.18}--\eqref{5.19b}.

\noindent\textit{(i) The energy splitting is exact for the
$c_1$-particle cell.}
The physical cell $R_1$ contains exactly $c_1$ particles,
with $l$ prescribed particles $x_1,\ldots,x_l \in R_1$
and $c_1 - l$ free particles $x_{l+1},\ldots,x_{c_1}$
integrated over $R_1^{c_1-l}$.
Under the one-cell restriction of
Conjecture~\ref{conj:cell}, there are no cross-cell
interaction terms, and the intra-cell energy splits as
the exact algebraic identity
\begin{equation*}
  W_1(x_1,\ldots,x_{c_1})
  \;=\;
  \underbrace{K\!\!\sum_{\substack{i,j=1\\i<j}}^{l}
    \Phi\!\left(\tfrac{|x_i-x_j|}{\epsilon}\right)}_{
    W_{Nl}\ \text{(prescribed--prescribed)}}
  +\;
  \underbrace{2K\sum_{j=1}^{l}\sum_{k=l+1}^{c_1}
    \Phi\!\left(\tfrac{|x_j-x_k|}{\epsilon}\right)}_{
    W^{*}_{Nl}\ \text{(prescribed--free)}}
  +\;
  \underbrace{K\!\!\sum_{\substack{i,k=l+1\\i<k}}^{c_1}
    \Phi\!\left(\tfrac{|x_i-x_k|}{\epsilon}\right)}_{
    W^{**}_{Nl}\ \text{(free--free)}}
\end{equation*}
Since $W_{Nl}$ does not involve the free integration
variables, it factors out as the shifted parameter
$B'$ in~\eqref{5.20}.
This splitting is exact and requires no further
approximation.

\noindent\textit{(ii) The replacement $c_1 \to N$ produces a
macroscopic reference cell, but the limit is unaffected.}
The replacement $n = 3c_1/2 \mapsto n = 3N/2$ in~\eqref{5.18}
defines new functions $C_{Nkl}(\rho,b)$ and $D_{Nkl}(\rho,b)$
on a reference cell $R$ with volume
\begin{equation*}
  r \;=\; \frac{ND\epsilon^3}{\rho} \;=\; \frac{1}{\rho},
\end{equation*}
using the scaling $ND\epsilon^3 = 1$.
This $R$ is a \emph{macroscopic} cell, not one of the small
partition cells of Assumption~\ref{assump:R_N}, and it
contains $N$ particles rather than the $c_1 \ll N$ of the
physical cell $R_1$.
The replacement therefore changes the finite-$N$ formula. We can still use \eqref{5.17} since its formula derivation doesn't depend on the size of $R$. 

Two facts justify it nonetheless.
First, formula~\eqref{5.16} for $g_{p,l}$ is derived
entirely in \emph{velocity space}: the energy constraint
is used to integrate out velocity variables, and the
positions enter only as parameters through $W_p(x)$.
The derivation uses neither the diameter nor the
volume of $R_p$, so~\eqref{5.16} holds for any cell
size, including a macroscopic one.
Second, we will know later, by Lemma~\ref{lemmaClimit},
\begin{equation*}
  C_{c_1,k,l}(\rho,b)
  \;\xrightarrow{\;c_1\to\infty\;}
  C(\rho,b)
  \;\xleftarrow{\;N\to\infty\;}
  C_{Nkl}(\rho,b),
\end{equation*}
where the limit $C(\rho,b)$ depends only on the local
density $\rho$ and the parameter $b$, not on the cell size
or particle count used to compute it.
This is the mathematical expression of \emph{thermodynamic
universality}: equilibrium properties in the thermodynamic
limit are functions of thermodynamic parameters alone.
The short-range assumption on $\Phi$ (Assumption~\ref{Phi},
decay faster than $r^{-3}$) ensures that the interaction
energy remains local even inside a macroscopic cell,
so the thermodynamic limit is well-defined.
The replacement $c_1 \to N$ therefore introduces an error
that vanishes in the subsequential limit of
Theorem~\ref{Thm:limit}, and the limiting object
$C(\rho,b)$ is the same regardless of which particle
count is used.

\noindent\textit{(iii) The combinatorial prefactor requires
$l \ll c_1$.}
The prefactor $\hat{K}_{N,l} = K_{c_1,l}/K_{c_1,0}$
in~\eqref{5.17} satisfies
\begin{equation*}
  \hat{K}_{N,l}
  \;=\;
  \frac{c_1^{\,3/2}\,
        \Gamma\!\left(\tfrac{3(c_1-1)}{2}\right)}
       {(c_1-l)^{3/2}\,
        \Gamma\!\left(\tfrac{3(c_1-l-1)}{2}\right)}
  \;=\;
  1 + O\!\left(\tfrac{l}{c_1}\right)
  \qquad \text{as } c_1 \to \infty,\ l \text{ fixed.}
\end{equation*}
After the replacement $c_1 \to N$ this reads
$1 + O(l/N)$.
The condition $l$ fixed and $N \to \infty$ is precisely
the BBGKY regime, and the error is absorbed into the
subsequential limit.

\subsection{Expression for $\phi_{Nl}$}
Finally, using \eqref{5.4}, \eqref{5.6}, and the normalization in \eqref{2.1}, we obtain
\begin{equation}
p_{Nl}
=
c_{Nl}\,\rho^l\,\mu(R_1)^l\,\kappa_{Nl},
\tag{5.21}\label{5.21}
\end{equation}
where
\[
c_{Nl}:=\frac{N!}{N^l(N-l)!}.
\]
Here $\rho^l$ reflects the local density contribution of the $l$ prescribed particles, while $\mu(R_1)^l$ is the macroscopic cell-volume factor. Also,
\[
c_{Nl}\to 1
\qquad\text{as }N\to\infty.
\]
Using \eqref{5.20}, \eqref{5.21}, and the definition \eqref{2.4}, we obtain
\begin{align}
\phi_{Nl}(x,u;\xi_1,v_1;\dots;\xi_{l-1},v_{l-1})
&=
\hat K_{N,l}
\left(\frac{B}{\pi}\right)^{3l/2}
\left[
\mathrm{pos}
\left(
1-\frac{B}{n}W_l-\frac{B}{n}V_{Nl}
\right)
\right]^{n-k_1}
\notag\\
&\quad\times
\rho^l
\left[
\frac{C_{Nk_0l}(\rho,B')}
     {C_{Nk_0l}(\rho,B)}
\right]^{n-k_0}
g_{Nk_1l}(x;\xi_1,\dots,\xi_{l-1};\rho,B').
\tag{5.22}\label{5.22}
\end{align}
Here
\begin{align*}
g_{Nkl}
&:=
\gamma_{Nkl}(x,x+\epsilon\xi_1,\dots,x+\epsilon\xi_{l-1};\rho,B'),\\
V_{Nl}
&:=
\sum_{j=0}^{l-1}|u-\bar u+v_j|^2
+
(N-l)^{-1}\left|\sum_{j=0}^{l-1}(u-\bar u+v_j)\right|^2,\\
W_l
&:=
\sum_{j,k=0}^{l-1}\Phi(|\xi_j-\xi_k|),
\qquad
\xi_0=0,
\qquad
v_0=0.
\end{align*}

Formula \eqref{5.22} is the finite-$N$ expression that will be used in the next section. It already has the structure of a local equilibrium distribution: the macroscopic fields enter through $\rho$, $\bar u$, and $\varepsilon$, while the microscopic dependence appears through the interaction term $W_l$, the fluctuation term $V_{Nl}$, and the normalized integrals $C_{Nkl}$ and $D_{Nkl}$.

To summarize, this section reduces the constrained marginal problem to an explicit finite-$N$ formula. The derivation separates a combinatorial part from a geometric one, then computes the geometric part through the Jacobian of the cell-wise constraints. The result is already close to the Gibbs-type form that appears later. The main remaining gap is the one-cell restriction in Conjecture \ref{conj:cell}, which is not proved here.

\section{The limiting function $\phi_l$}\label{Section:limit}
In this section, we try to show that $\phi_{Nl}$ in the form of \eqref{5.22} converges subsequentially to certain $\phi_l$ which is the Gibbs distribution. Recall that the dependence on the cell $R$ now reduces to the dependence on $\rho$ in \eqref{5.19a} if we can take the hydrodynamic limit, by van Hove theorem \cite{Ruelle1969}. Without loss of generality, we can treat $R$ as fixed.

The usual assumed form for $\phi_l$ Gibbsian (see \cite[Chapter 6]{BornGreen1949})
reduces in our notation to
\begin{equation}
  \phi_l = (A/\pi)^{3l/2}
  \exp\!\left\{
    -A\!\left[
      \sum_{i=0}^{l-1} \|u + v_i - \bar{u}\|^2 + 2\mathcal{X}_l
    \right]
  \right\}
  \cdot \rho^l g_l(\xi_1, \dots, \xi_{l-1}; \rho, A),
  \tag{6.1}\label{6.1}
\end{equation}
where
\[
  2\mathcal{X}_l = K\sum_{\substack{i,k=0\\ i \neq k}}^{l-1} \Phi(\|\xi_i - \xi_k\|),
  \qquad
  v_0 = \xi_0 = 0,
\]
and $\bar{u}(t;x)$ and $\rho(t;x)$ have their usual physical meaning. $g_l$ is the correlation and will be explicitly computed in Section \ref{Sectiongl}.  The constant $A(t;x)$ is given by
\[
  A(t;x) = \frac{1}{2kT(t;x)},
\]
where $T(t;x)$ denotes the local temperature.
In what follows we regard $A$, $\rho$, and $\bar{u}$ as constants.

\textbf{Remark:} Previous, we didn't treat $\rho,\bar{u}$ as constant. Now Morrey is moving to a uniform equilibrium status: time independent, and since our $\phi_l$ is free of $x$, one can just consider fixed $x$ so that $\rho, \bar{u}$ then become constant. (This inequivalency of $x$ and $\xi$ still confused me though.)

We also require that
\begin{equation}
  \lim_{\xi_i \to \infty, \, \xi_i - \xi_{i-1} \to \infty}
  g_l(\xi_1, \dots, \xi_{l-1}; \rho, A) = 1.
  \tag{6.2}\label{6.2}
\end{equation}
\textbf{Remark:}
This is the physical requirement that when all particles in a group are infinitely separated, the correlations between them vanish and their joint probability distribution factorizes into the product of independent single--particle distributions.

In equation \eqref{6.1} the function $g_l$ represents the \emph{correlation correction}
to the product of one--particle Gibbs distributions. 
If the particles were completely independent, we would have $g_l \equiv 1 $ and $ \phi_{l+1}$ sort of $\phi_l\cdot\phi_1$. 
Thus, $g_l \rightarrow 1$ measures the deviation from independence
arising from interparticle correlations. The concrete form of $g_l$ will be derived in Section \ref{Sectiongl}.

\begin{theorem}[6.1]\label{Thm:limit}
Given any subsequence of the integers $N$, We can extract a further sequence such that $C_{Nkl}(\rho,b)$ in \eqref{5.22} converge uniformly to $C(\rho,b)$. Suppose for $\rho_0=ND\epsilon^3/r$ and $B_0=\frac{3}{4\varepsilon_1}$, $C(\rho_0,B_0)>0$ and $\frac{\partial C}{\partial b}(\rho_0,B_0)$ exists. And we assume there exist positive constants $\gamma,\delta,m_{kl},M_{kl}$ and $N_{kl}$ and such that for $g_{Nkl}$ in \eqref{5.22}
\begin{equation}
m_{kl} 
\le 
g_{N kl}(x,\xi_1,\ldots,\xi_{l-1};\rho,b)
\le 
M_{kl},
\qquad 
C(\rho,b)\ge \gamma ,
\tag{6.3}\label{6.3}
\end{equation}
whenever
\[
N>N_{kl}, 
\qquad 
|\rho-\rho_0|\le \delta,
\qquad 
|b-B_0|\le \delta,
\]
and all $x_1,\ldots,x_l \in R\times\cdots\times R$, with $N$ in the chosen subsequence.

Then a further subsequence, let's still call it $N$ for simplicity, may be selected such that, for each $l$,
\[
\phi_{Nl}(x,u;\xi_1,v_1;\ldots;\xi_{l-1},v_{l-1};\rho_0,B_0)
\]
converges to a function $\phi_l$, uniformly for
$
x \text{ on any } R' \subset R^{(0)}, \qquad
\xi_1,\ldots,\xi_{l-1} \text{ on any bounded domain } G,
$ and $
u,v_1,\ldots,v_{l-1} \text{ on any bounded } \Gamma,
$
with
\begin{equation}
\phi_{Nl}\exp\{\beta(V_l+W_l)\}
\tag{6.4}\label{6.4}
\end{equation}
uniformly bounded uniformly in $N$ for some $\beta>0$, where
\[
V_l=\sum_{j=0}^{l-1}|u-\bar u+v_j|^2,
\qquad
v_0=0 .
\]

What's more, the limiting functions $\phi_l$ are of the form \eqref{6.1} with
\begin{equation}
A = B_0\left(1 - \frac{B_0}{C(\rho_0,B_0)}\frac{\partial C(\rho_0,B_0)}{\partial b}\right).
\tag{6.5}\label{6.5}
\end{equation}
\end{theorem}
\textbf{Remark :} The boundedness of $g_{Nkl}$ \eqref{6.3} is not rigorously proved yet.

\subsection{Lemmas}
To prove this theorem, one needs certain estimations. The notation used below follows that introduced in \eqref{5.18} and (5.19).

\begin{lemma}[6.2]\label{lem:bounds}
    For all integers $n,k,l$ and points $x_1,\ldots,x_l \in R$, let $m = -\min \Phi(r).$
Then, for $\rho$ small enough,
\begin{equation}
\frac{1}{12}\left[1-\frac{2}{3}\phi(h)\right]^n
\le
D^{\,n-k}_{N kl}(x_1,\ldots,x_l;\rho,b)
\le
\left(1+\frac{2}{3}bKL\right)^{n-k}
\exp\!\left\{\frac{l(l-1)mKb}{1+2bKL/3}\right\},
\tag{6.6}\label{6.6}
\end{equation}
and
\begin{equation}
\frac{1}{12}\left[1-\frac{2}{3}\phi(h)\right]^n
\le
C^{\,n-k}_{N kl}(\rho,b)
\le
\left(1+\frac{2}{3}bKL\right)^{n-k},
\qquad N>l .
\tag{6.7}\label{6.7}
\end{equation}
Here
\[
h=\frac{\rho \tau_0(b)}{D},
\qquad
\phi(h)=\sum_{\nu=1}^{\infty}\frac{h^\nu}{\nu(\nu+1)},
\qquad
0\le h\le1 .
\]
If $\Phi(r)\le0$ for $r\ge r_0$, then
\begin{equation}
C^{\,n-k}_{N kl}(\rho,b)\ge e^{-n\gamma},
\qquad
D^{\,n-k}_{N kl}\ge e^{-n\gamma},
\tag{6.8}\label{6.8}
\end{equation}
where
\[
\gamma :=\tfrac23\phi(h_0),
\qquad
h_0:=\frac{4\pi r_0^3\rho}{3D},
\qquad
h_0\le1 .
\]

\end{lemma}
\begin{proof}
We first prove the lower bound in \eqref{6.6}. As in Morrey, write
\[
\Phi=\Phi_0-\Phi_1
\]
with the notation from Section 8. In particular,
\[
W^{**}_{0pl}=W^{***}_{0Nl}\big|_{N\mapsto p},\qquad W^{**}_{0ll}=0.
\]
Then for $l\le p < N$ we have
\begin{equation}
W^{**}_{0,p+1,l}+W^{*}_{0,p+1,l}
=
2K\sum_{i=1}^p \Phi_0\!\left(\frac{|x_i-x_{p+1}|}{\epsilon}\right)
+
W^{**}_{0pl}+W^{*}_{0pl}.
\tag{6.9}\label{6.9}
\end{equation}

Now recall that in our notation
\[
ND\epsilon^3=pr.
\]
Hence
\begin{align}
&r^{-1}\int_R
\left[
1-\exp\left\{-2Kb\sum_{i=1}^p \Phi_0\!\left(\frac{|x_i-x_{p+1}|}{\epsilon}\right)\right\}
\right]dx_{p+1}
\notag\\
&\le
\sum_{i=1}^p
r^{-1}\int_R
\left[
1-\exp\left\{-2Kb\,\Phi_0\!\left(\frac{|x_i-x_{p+1}|}{\epsilon}\right)\right\}
\right]dx_{p+1}
\notag\\
&\le
\frac{p}{N}\cdot \frac{\rho}{D}
\int_{\mathbb R^3}
\left[1-\exp\{-2Kb\,\Phi_0(|\xi|)\}\right]\,d\xi
=
\frac{p}{N}h.
\tag{6.10}\label{6.10}
\end{align}
Here $h$ is the quantity introduced earlier in the statement of the lemma.

Using \eqref{6.9} and \eqref{6.10} with $p=N-1$, we obtain
\begin{align*}
&r^{-N+l}\int_{R\times\cdots\times R}
\exp\{-bW^{**}_{0Nl}-bW^{*}_{0Nl}\}\,dx_{l+1}\cdots dx_N
\\
&=
r^{-N+l+1}\int_{R\times\cdots\times R}
\Bigg[
r^{-1}\int_R
\exp\left\{-2Kb\sum_{i=1}^{N-1}\Phi_0\!\left(\frac{|x_i-x_N|}{\epsilon}\right)\right\}dx_N
\Bigg]
\\
&\hspace{4cm}\times
\exp\{-bW^{**}_{0,N-1,l}-bW^{*}_{0,N-1,l}\}\,dx_{l+1}\cdots dx_{N-1}
\\
&\ge
\left[1-\frac{N-1}{N}h\right]
r^{-N+l+1}
\int_{R\times\cdots\times R}
\exp\{-bW^{**}_{0,N-1,l}-bW^{*}_{0,N-1,l}\}\,dx_{l+1}\cdots dx_{N-1}.
\end{align*}
Repeating this argument successively for $x_N,x_{N-1},\dots,x_{l+1}$ gives
\begin{equation}
r^{-N+l}\int_{R\times\cdots\times R}
\exp\{-bW^{**}_{0Nl}-bW^{*}_{0Nl}\}\,dx_{l+1}\cdots dx_N
\ge
\prod_{j=l}^{N-1}\left(1-\frac{j}{N}h\right)
=: \Pi_N(h).
\tag{6.11}\label{6.11}
\end{equation}

Next, estimate $\Pi_N(h)$ from below. We write
\begin{equation*}
-\log \Pi_N(h)
=
\sum_{\nu=1}^\infty \frac{h^\nu}{\nu}\sum_{j=0}^{N-1}\left(\frac{j}{N}\right)^\nu
\le N\phi(h),
\tag{6.12}\label{6.12}
\end{equation*}
so that
\[
\Pi_N(h)\ge e^{-N\phi(h)}=:e^{-n\beta},
\qquad
\beta=\frac{2}{3}\phi(h).
\]

Now choose $s$ so that
\[
\beta<bs<1,
\]
and let $\Sigma$ denote the set of $(x_{l+1},\dots,x_N)$ for which
\[
n^{-1}(W^{**}_{0Nl}+W^{*}_{0Nl})\le s.
\]
Then
\begin{align*}
D^{\,n-k}_{Nkl}
&\ge D^{\,n-k}_{0Nkl}
\notag\\
&\ge
r^{-N+l}\int_\Sigma
\left[
\exp\{bW^{**}_{0Nl}+bW^{*}_{0Nl}\}
\cdot
\Bigl(\mathrm{pos}\bigl(1-\tfrac{b}{n}W^{**}_{0Nl}-\tfrac{b}{n}W^{*}_{0Nl}\bigr)\Bigr)^{n-k}
\right]
\notag\\
&\times
\exp\{-bW^{**}_{0Nl}-bW^{*}_{0Nl}\}\,dx_{l+1}\cdots dx_N\\
&\ge
\bigl[e^{bs}(1-bs)\bigr]^n
\, r^{-N+l}\int_\Sigma
\exp\{-bW^{**}_{0Nl}-bW^{*}_{0Nl}\}\,dx_{l+1}\cdots dx_N
\notag\\
&\ge
\bigl(e^{-nbs-n\beta}-1\bigr)(1-bs)^n.
\tag{6.13}\label{6.13}
\end{align*}

On the right-hand side, set
\[
bs=\beta+\frac{y}{n}.
\]
Then \eqref{6.13} becomes
\begin{equation}
(1-\beta)^n(e^y-1)\left[1-\frac{1}{n}\frac{y}{1-\beta}\right]^n,
\qquad
0<y<n(1-\beta).
\tag{6.14}\label{6.14}
\end{equation}
For fixed $y$, the third factor increases with $n$, so \eqref{6.14} is bounded below by
\[
(1-\beta)^n\left[y-\frac{y^2}{1-\beta}\right],
\qquad
0<y<1-\beta.
\]
This gives the desired left-hand inequality in \eqref{6.6}.

The left-hand inequality in \eqref{6.7} for $l=0$ is just a special case of \eqref{6.6}. For general $l$, we use
\begin{equation}
C_{Nkl}(\rho,b)
=
C_{N-l,\,k-3l/2,\,0}\bigl(\rho(1-l/N),b\bigr),
\tag{6.15}\label{6.15}
\end{equation}
since
\[
(N-l)D\epsilon^3=(1-l/N)r
\]
in our notation. Hence the same lower bound follows for all $l$.

For the remaining lower bound involving the $D$'s, let $\Sigma_0\subset R^{N-l}$ be the set on which
\[
\frac{|x_j-x_k|}{\epsilon}\ge r_0
\]
for all $j,k=l+1,\dots,N$, and also for $j=1,\dots,l$, $k=l+1,\dots,N$. By choosing the variables one by one, we obtain
\begin{equation}
r^{-N+l}\mu(\Sigma_0)
\ge
\left(1-\frac{lh_0}{N}\right)\cdots
\left(1-\frac{(N-1)h_0}{N}\right)
\ge
\Pi_N(h_0),
\tag{6.16}\label{6.16}
\end{equation}
where
\[
h_0=\frac{4}{3}\pi r_0^3\epsilon^3r^{-1}N=\frac{4\pi r_0^3\rho}{3D}.
\]
Now the conclusion follows from \eqref{6.12}, since on $\Sigma_0$ the integrand in $D^{\,n-k}_{Nk}$ is bounded below by $1$.
\end{proof}

\begin{lemma}[6.2]
    The functions $C_{N kl}(\rho,b) \;\text{and}\; D_{N kl}(x_1;\cdots;x_l;\rho,b)$ are non–negative and convex with regard to $b$. Furthermore,
\begin{align}
b\,D^{-1}_{N kl}\frac{\partial D_{N kl}}{\partial b}
&=
1-D^{\,n+k}_{N kl}\,D^{\,n-k-1}_{N,k+1,l},
\tag{6.17a}\label{6.17a}\\[6pt]
b\,C^{-1}_{N kl}\frac{\partial C_{N kl}}{\partial b}
&=
1-C^{\,n+k}_{N kl}\,C^{\,n-k-1}_{N,k+1,l},\tag{6.17b}\label{6.17b}\\[6pt]
\left|\rho\,C^{-1}_{Nkl}\,\frac{\partial C_{Nkl}}{\partial\rho}\right|
&\le
\frac{2}{9}L_1 b
-\left[\frac{2}{9}L_1 b+\frac{\delta+1}{3}\right]
b\,C^{-1}_{Nkl}\,\frac{\partial C_{Nkl}}{\partial b}.
\tag{6.17c}\label{6.17c}
\end{align}
\end{lemma}

\begin{proof}
We prove the statements for the functions $C_{Nkl}$; the argument for $D_{Nkl}$ is analogous.
To simplify notation we suppress the indices $N$ and $l$ throughout the proof.

Differentiating the representation of C in \eqref{5.19a} with respect to $b$ and dividing by $n-k$
gives
\begin{equation}
C_k^{\,n-k-1} C_{k,b}
=
r^{-N+l}
\int_{R\times\cdots\times R}
\left[\operatorname{pos}\!\left(1-\frac{b}{n}W_{Nl}^{**}\right)\right]^{n-k-1}
\left(-\frac{1}{n}W_{Nl}^{**}\right)
\,dx_{l+1}\cdots dx_N ,
\tag{6.18}\label{6.18}
\end{equation}
which can also be written as
\[
C_k^{\,n-k-1} C_{k,b}
=
b^{-1}\left(C_k^{\,n-k}-C_{k+1}^{\,n-k-1}\right).
\]

Differentiating \eqref{6.18} once more and multiplying by $C_k^{\,n-k}$ yields
\begin{align}
C_k^{\,2n-2k-1} C_{k,bb}
+
(n-k-1)\bigl(C_k^{\,n-k-1}C_{k,b}\bigr)^2
&=
C_k^{\,n-k} r^{-N+l}
\int_{R\times\cdots\times R}
(n-k-1)
\left[\operatorname{pos}\!\left(1-\frac{b}{n}W_{Nl}^{**}\right)\right]^{n-k-2}
\nonumber\\
&\qquad\times
\left(-\frac{1}{n}W_{Nl}^{**}\right)
dx_{l+1}\cdots dx_N .
\tag{6.19}\label{6.19}
\end{align}

Now in \eqref{6.19} we replace $C_k^{\,n-k}$ on the right-hand side by an
integral representation of the form \eqref{5.19a} in variables
$y_{l+1},\dots,y_N$. 
We then symmetrize the resulting expression in the variables $x$ and $y$.
Finally we transpose the second term on the left-hand side of \eqref{6.19}
to the right-hand side and replace it by a product of integrals of the form
\eqref{6.18}, one involving the $x$ variables and the other the $y$ variables.
After simplification we obtain

\begin{align}
C_k^{\,2n-2k-1} C_{k,bb}
&=
\frac{n-k-1}{2}\, r^{-2N+2l}
\int_{R\times\cdots\times R}
\left[\operatorname{pos}\!\left(1-\frac{b}{n}W_{Nl}^{**}(x)\right)\right]^{n-k-2}
\nonumber\\
&\qquad\times
\left[\operatorname{pos}\!\left(1-\frac{b}{n}W_{Nl}^{**}(y)\right)\right]^{n-k-2}
\nonumber\\
&\qquad\times
\Bigg\{
\left(1-\frac{b}{n}W_{Nl}^{**}(x)\right)\frac{W_{Nl}^{**}(y)}{n}
-
\left(1-\frac{b}{n}W_{Nl}^{**}(y)\right)\frac{W_{Nl}^{**}(x)}{n}
\Bigg\}^2
dx\,dy .
\tag{6.20}\label{6.20}
\end{align}

The integrand in \eqref{6.20} is non–negative, and therefore $C_k$ is convex in $b$.
\eqref{6.17b} follows directly from \eqref{6.18}.

Since $ND\epsilon^3=\rho r$ in our notation, the only way $\rho$ can vary
(for fixed $N$ and $R$) is through $\epsilon$. Thus
\[
\rho\,\frac{\partial C_k}{\partial \rho}
=
\frac{1}{3}\epsilon\,\frac{\partial C_k}{\partial \epsilon}.
\]

Carrying out the differentiation gives
\begin{align}
\rho C_k^{\,n-k-1}C_{k,\rho}
&=
\frac{1}{3} r^{-N+l}
\int_{R\times\cdots\times R}
\left[\operatorname{pos}\!\left(1-\frac{b}{n}W_{Nl}^{**}\right)\right]^{n-k-1}
\nonumber\\
&\qquad\times
\left\{
b\,\frac{1}{n}
\sum_{i=1}^{l}\sum_{p=l+1}^{N}
|x_i-x_p|\,\Phi'(|x_i-x_p|/\epsilon)
\right\}
dx_{l+1}\cdots dx_N .
\tag{6.21}\label{6.21}
\end{align}

Using Theorem~\ref{thm:LNbound} together with \eqref{5.18} we obtain the estimate
\begin{align*}
\left|\rho C_k^{\,n-k-1}C_{k,\rho}\right|
&\le
\frac{\delta+1}{3} r^{-N+l}
\int_{R\times\cdots\times R}
\left[\operatorname{pos}\!\left(1-\frac{b}{n}W_{Nl}^{**}\right)\right]^{n-k-1}
\left(\frac{b}{n}W_{Nl}^{**}\right)
dx
\nonumber\\
&\qquad
+
\frac{2}{9} L_1 b\,C_{k+1}^{\,n-k-1},
\end{align*}
from which the desired inequality \eqref{6.17c} follows.
\end{proof}

\begin{lemma}[6.3]\label{lemmaconvex}
    Let $f_p(x)$ be convex functions on $[a,b]$ which converge
uniformly to a function $f(x)$.
Assume $f'(x_0)$ exists for some $x_0\in(a,b)$ and that
\[
x_0+h_p\in[a,b], \qquad h_p\to0 .
\]
Then
\[
\lim_{p\to\infty}
\frac{f_p(x_0+h_p)-f_p(x_0)}{h_p}
=
f'(x_0).
\]
What's more, if $f_p'(x_0)$ exists for all $p$, then
\[
\lim_{p\to\infty}f_p'(x_0)=f'(x_0).
\]
\end{lemma}

\begin{lemma}[6.4]\label{lemmaClimit}
    From any subsequence of the integers $N$ one may extract a further subsequence for which the functions $C_{N kl}(\rho,b)$ converge uniformly on compact subsets of the quadrant $\rho>0, \;b\ge0$ to a \emph{limiting function} $C(\rho,b)$.

The limit function $C$ is continuous, convex in $b$, and satisfies
a uniform Lipschitz bound on compact subsets of $\rho>0$, $b>0$.

Moreover, if
\[
\rho_0>0, \qquad b_0>0,
\qquad C(\rho_0,b_0)>0,
\]
and if $\frac{\partial C}{\partial b}(\rho_0,b_0)$ exists, then

\begin{equation}
\lim_{N\to\infty}
\frac{C^{\,n-k}_{N kl}(\rho_0,b_0)}
{C^{\,n}_{N0l}(\rho_0,b_0)}
=
\left[
1-\frac{b_0}{C(\rho_0,b_0)}\frac{\partial C}{\partial b}(\rho_0,b_0)
\right]^k .
\tag{6.22}\label{6.22}
\end{equation}
Furthermore,
\begin{equation}
\lim_{N\to\infty}
\left[
\frac{C_{N kl}(\rho_0,b_N')}
{C_{N kl}(\rho_0,b_0)}
\right]^{n-k}
=
\exp\!\left\{
\frac{\beta}{C(\rho_0,b_0)} \frac{\partial C}{\partial b}(\rho_0,b_0)
\right\},
\tag{6.23}\label{6.23}
\end{equation}
whenever
\[
\lim_{N\rightarrow\infty} N(b_N'-b_0)=\beta .
\]
\end{lemma}

\begin{proof}
The functions $C_{Nkl}(\rho,b)$ are convex in $b$ for $b\ge0$ when $\rho>0$,
and they satisfy the bounds given in \eqref{6.7} and \eqref{6.8}. 
Hence the quantities $C_{Nkl,b}$ are bounded on every compact subset of the
open quadrant $\rho>0$, $b>0$.

Multiplying relation \eqref{6.17b} by $C_{Nkl}$ shows that the derivatives
$C_{Nkl,\rho}$ are also bounded. 
Furthermore, the bounds \eqref{6.7} imply that the functions $C_{Nkl}$ are
equicontinuous on compact subsets of the axis $b=0$, $\rho>0$.

Therefore, for each pair $(k,l)$ there exists a subsequence along which
$C_{Nkl}$ converges uniformly on compact subsets to a limit function,
which we denote by $C_{kl}$.

A direct inspection of the definition \eqref{5.19a} shows that
\begin{equation}
C_{Nkl}(\rho,b)
=
C_{N-l,k-\frac{3l}{2},0}\!\left(\rho\!\left(1-\frac{l}{N}\right),b\right).
\tag{6.24}\label{6.24}
\end{equation}
Consequently $C_{kl}=C_{k-\frac{3l}{2},0}.$

Next we introduce
\begin{equation}
f_{N,l}(k;\rho,b)
=
\log
\left[
\frac{C_{Nkl}^{\,n-k}(\rho,b)}{C_{N0l}^{\,n}(\rho,b)}
\right].
\tag{6.25}\label{6.25}
\end{equation}

Differentiating with respect to $k$ yields
\begin{align}
\frac{\partial f_{N,l}}{\partial k}
&=
C_{Nkl}^{\,-n+k}\,
r^{-N+l}
\int_{R\times\cdots\times R}
\left[
\operatorname{pos}
\!\left(
1-\frac{b}{n}W_{Nl}^{**}
\right)
\right]^{n-k}
\nonumber\\
&\qquad\times
\left[
-\log
\!\left(
1-\frac{b}{n}W_{Nl}^{**}
\right)
\right]
dx
\tag{6.26}\label{6.26}
\end{align}

and therefore
\[
\frac{\partial f_{N,l}}{\partial k}
\ge
-\log(1+2L_1b/3),
\]
for all $k$, by Theorem~\ref{thm:LNbound}.

Since $C_{Nkl}$ is non-increasing in $k$ (by Hölder's inequality), 
relations \eqref{6.25}–\eqref{6.26} imply
\begin{align*}
-k\log(1+2L_1b/3)
&\le
f_{N,l}
\le
-k\log C_{N0l},
\qquad k\ge0,
\\
-k\log C_{N0l}
&\le
f_{N,l}
\le
-k\log(1+2L_1b/3),
\qquad k\le0.
\tag{6.27}\label{6.27}
\end{align*}

In particular, if $C_{00}(\rho_0,b_0)>0$ then the functions $C_{k0}$ are all
identical. The same conclusion holds if $0$ is replaced by any negative
value $k_2$. Since $C_{Nkl}$ is non-increasing in $k$, it follows that
all $C_{k0}$ coincide whenever one of them vanishes.

Differentiating \eqref{6.26} again with respect to $k$ and proceeding
as in \eqref{6.19}–\eqref{6.20}, we obtain
\begin{align*}
C_{Nkl}^{\,2n-2k}
\frac{\partial^2 f_{N,l}}{\partial k^2}
&=
\frac{1}{2}\,
r^{-2N+2l}
\int_{R\times\cdots\times R}
\left[
\operatorname{pos}
\!\left(
1-\frac{b}{n}W_{Nl}^{**}(x)
\right)
\right]^{n-k}
\nonumber\\
&\qquad\times
\left[
\operatorname{pos}
\!\left(
1-\frac{b}{n}W_{Nl}^{**}(y)
\right)
\right]^{n-k}
\nonumber\\
&\qquad\times
\left\{
\log\!\left(1-\frac{b}{n}W_{Nl}^{**}(x)\right)
-
\log\!\left(1-\frac{b}{n}W_{Nl}^{**}(y)\right)
\right\}^{2}
dx\,dy .
\end{align*}

Hence $f_{N,l}$ is convex in $k$.
Using \eqref{6.17b} and \eqref{6.25} we also have
\begin{equation}
f_{N,l}(k;\rho,b)
=
\sum_{i=0}^{k-1}
\log
\left(
1
-
b\,C_{Ni,l}^{-1}\frac{\partial C_{Ni,l}}{\partial b}
\right),
\tag{6.28}\label{6.28}
\end{equation}
for every integer $k$.

Evaluating this expression at $(\rho_0,b_0)$ and passing to the limit
$N\to\infty$, we obtain
\begin{equation}
    \lim_{N\to\infty}
f_{N,l}(k;\rho_0,b_0)
=
k\log
\left[
1-
\frac{b_0 \frac{\partial C}{\partial b}(\rho_0,b_0)}{C(\rho_0,b_0)}
\right]. \tag{6.29}\label{6.29}
\end{equation}

From this relation equation \eqref{6.22} follows, while \eqref{6.23}
is a consequence of the convexity of the functions $C_{Nkl}$ together
with Lemma \ref{lemmaconvex}.
\end{proof}

\subsection{Proof of Theorem \ref{Thm:limit} with certain gaps}
\begin{proof}
We shall leave a few gaps in the end of this proof but present the main ideas.

From \eqref{6.3}, it follows immediately that the functions $D_{Nkl}$
converge uniformly to $C$ as $N$ runs through the given subsequence.
Moreover, \eqref{6.22} and \eqref{6.23} remain valid with $C_{Nkl}$
replaced by $D_{Nkl}$. In \eqref{6.23}, the parameter $\beta$ may depend
on $x,\xi_1,\dots,\xi_{l-1}$, and uniform convergence is preserved.

Accordingly,
\begin{equation}
\lim_{N\to\infty}
\left[
\frac{C_{Nk0}(\rho_0,B_0')}{C_{Nk0}(\rho_0,B_0)}
\right]^{n-k}
=
\exp\!\left\{
\frac{B_0^2 (W_l + V_l)}{C(\rho_0,B_0)}\,\frac{\partial C}{\partial b}(\rho_0,B_0)
\right\}
=
\lim_{N\to\infty}
\left[
\frac{D_{Nkl}(x,\xi_1,\dots,\xi_{l-1};\rho_0,B_0')}
{D_{Nkl}(x,\xi_1,\dots,\xi_{l-1};\rho_0,B_0)}
\right]^{n-k},
\tag{6.30}\label{6.30}
\end{equation}
where $B_0'$ is defined as in \eqref{5.20}. The convergence is uniform on
restricted sets.

From our hypotheses and \eqref{5.22}, \eqref{6.23}, etc., it follows that
\eqref{6.4} is uniformly bounded for some $\beta>0$. From \eqref{6.30},
we deduce that $\phi_{Nl}/\tilde{\phi}_{Nl} \to 1$ uniformly if
$\tilde{\phi}_{Nl}$ is defined as in \eqref{5.22} with $B'$ replaced by $B_0$
in the last two factors.

From \eqref{6.22}, \eqref{6.23}, and standard arguments, it follows that
the coefficient of
$g_{Nkl}(x,\xi_1,\dots,\xi_{l-1};\rho_0,B_0)$ in $\phi_{Nl}$ converges
uniformly on restricted sets to
\begin{equation}
\left(\frac{A_0}{\pi}\right)^{3l/2}
\exp\{-A_0(V_l + W_l)\}\,\rho^l,
\qquad
A_0 = B_0\left(1 - \frac{B_0}{C}\frac{\partial C}{\partial b}\right).
\tag{6.31}\label{6.31}
\end{equation}

It remains to consider the functions $g_{Nkl}$. Viewing them first as
functions of $x_1,\dots,x_l$ and using symmetry in the remaining variables,
we obtain
\begin{align}
g_{Nkl,i}^{\alpha}
&=
-2K B_0\Bigl(1-\frac{k}{n}\Bigr)(N-l)
C_{Nk0}^{-n+k}(\rho_0,B_0)
\nonumber\\
&\quad\times r^{-N+l}
\int_{R\times\cdots\times R}
\Bigl[
\operatorname{pos}\Bigl(
1-\frac{B_0}{n}\Omega_{Nl}
-\frac{B_0}{n}W^{**}_{N,l+1}
-\frac{B}{n}W^{*}_{N,l+1}
\Bigr)
\Bigr]^{n-k-1}
\nonumber\\
&\quad\times
\Phi'(|x_{l+1}-x_i|/\epsilon)
\frac{x_{l+1}^{\alpha}-x_i^{\alpha}}{|x_{l+1}-x_i|}
\,dx,
\tag{6.32}\label{6.32}
\end{align}
where
\[
\Omega_{Nl}
=
2 \sum_{j=1}^l \Phi(|x_j - x_{l+1}|/\epsilon).
\]

Let $H>0$ be fixed and define $\Sigma_{1N} = \{\Omega_{Nl} > H\}$ and
$\Sigma_{2N} = R \setminus \Sigma_{1N}$.

For $x_{l+1} \in \Sigma_{1N}$, the integrand in \eqref{6.32} satisfies
\begin{align}
&\le
A_{n,k+1}
\Bigl[
\operatorname{pos}\Bigl(
1-\frac{b}{n}W^{**}_{N,l+1}
-\frac{b}{n}W^{*}_{N,l+1}
\Bigr)
\Bigr]^{n-k-1}
\nonumber\\
&\quad\times
\exp\{-\beta_N \Omega_{Nl}\}
\cdot \Phi'(|x_i - x_{l+1}|/\epsilon)
\nonumber\\
&\le
\varepsilon(H)\,A_{n,k+1}
\Bigl[
\operatorname{pos}\Bigl(
1-\frac{b}{n}W^{**}_{N,l+1}
-\frac{b}{n}W^{*}_{N,l+1}
\Bigr)
\Bigr]^{n-k-1},
\tag{6.33}\label{6.33}
\end{align}
where $\varepsilon(H)\to 0$ as $H\to\infty$, uniformly.

Here
\begin{equation}
\beta_N \le 1 + \frac{2}{3}KL_1b + \frac{K(l+1)mb}{2n},
\tag{6.34}\label{6.34}
\end{equation}
and
\begin{equation}
\frac{1}{n}\bigl(W^{**}_{N,l+1} + W^{*}_{N,l+1} + W_{N,l+1}\bigr)
=
\frac{2}{3}\frac{K}{N}
\sum_{i,k=1}^N \Phi(|x_i-x_k|/\epsilon).
\tag{6.35}\label{6.35}
\end{equation}

Thus the contribution from $\Sigma_{1N}$ is negligible.

For $\Sigma_{2N}$, we rewrite the integral as
\begin{align}
&-2K B_0\Bigl(1-\frac{k}{n}\Bigr)\Bigl(1-\frac{l}{N}\Bigr)
\frac{C_{N,k+1,0}(\rho_0,B_0)}{C_{Nk0}(\rho_0,B_0)}
\frac{\rho}{D}
\nonumber\\
&\quad\times
\int_{R_{N,x}-\Sigma(H)}
g_{N,k+1,l+1}(x,\xi_1,\dots,\xi_l;\rho_0,B')
\Bigl[
\operatorname{pos}\bigl(1-(B_0/n)\Omega_l\bigr)
\Bigr]^{n-k-1}
\nonumber\\
&\quad\times
\left[
\frac{C_{N,k+1,0}(\rho_0,B')}{C_{N,k+1,0}(\rho_0,B_0)}
\right]^{n-k-1}
\Phi'(|\xi_l-\xi_i|)
\frac{\xi_l^\alpha - \xi_i^\alpha}{|\xi_l - \xi_i|}
\,d\xi_l,
\tag{6.36}\label{6.36}
\end{align}
where
\[
B' = \frac{B_0}{1 - n^{-1} B_0 \Omega_l}.
\]

For fixed $\xi_1,\dots,\xi_{l-1}$, we have $B'\to B_0$ uniformly, and
\[
n(B' - B_0) \to B_0 \Omega_l.
\]

Thus the coefficient converges to
\begin{equation}
\Phi'(|\xi_i - \xi_l|)
\frac{\xi_i^\alpha - \xi_l^\alpha}{|\xi_i - \xi_l|}
\exp\{-A_0 \Omega_l\}.
\tag{6.37}\label{6.37}
\end{equation}

Combining \eqref{6.33}, \eqref{6.36}, and \eqref{6.37}, we conclude that $g_{Nkl\xi_j}^\alpha$ are uniformly bounded.

Similarly, we can hope that $g_{Nkl x}$ are also bounded. However, this has not been proved and remains a \emph{gap}. We shall leave it as a conjecture.
\begin{conjecture}\label{conj:deri}
    Under the same assumption as Theorem \ref{Thm:limit}, the $x$-derivatives of $g_{Nkl}$ are uniformly bounded on bounded sets, uniformly in N for fixed $k$ and $l$.
\end{conjecture}

Hence a further subsequence may be extracted, yielding limiting functions $g_{kl}$. 
\end{proof}

\section{A Formal Solution of the correlation $g_l$ in \eqref{6.1}}\label{Sectiongl}
Assuming the Gibbs-type limiting form \eqref{6.1}, the next goal is to determine the correlation functions $g_l$.

Assuming the required uniform convergence of $\phi_{Nl}$ and their derivatives and Conjecture \ref{conj:pi}, we should expect that the functions $\phi_l$
satisfy the infinite sequence of equations obtained formally
by letting $N \to \infty$ and $\epsilon \to 0$ in \eqref{2.5}. More precisely speaking, we want our $\phi_{NN}$ to be \emph{local equilibrium distributions}, which is a strong assumption and guarantees the derivation of fluid dynamics.\cite{DeMasiEtAl1984}

The limit of \eqref{2.5} is an infinite sequence of equations:
\begin{align*}\label{7.1}\tag{7.1}
    &B_{l0}^\alpha(\xi)\,\phi_{l,u^\alpha}
\;+\; \sum_{i=1}^{\,l-1}\big[\,u^\alpha \phi_{l,\xi_i^\alpha}+B_{li}^\alpha(\xi)\,\phi_{l,v_i^\alpha}\big]\\    =&
-
\int_{-\infty}^{\infty}
\bigg\{
\Phi'\big(|\xi_1|\big)\,\frac{\xi_1^\alpha}{|\xi_1|}
\;\phi_{l+1,\;u^\alpha}
\;+\;
\sum_{i=1}^{\,l-1}
\bigg[
\Phi'\big(|\xi_i-\xi_1|\big)\,\frac{\xi_i^\alpha-\xi_1^\alpha}{|\xi_i-\xi_1|}
-
\Phi'\big(|\xi_1|\big)\,\frac{\xi_1^\alpha}{|\xi_1|}
\bigg]\,
\phi_{,l+1,\;v_i^\alpha}
\bigg\}\,d\xi_1\,dv_1,
\end{align*}

\vspace{0.5em}
\noindent
The limiting equations \eqref{7.1} are seen to possess an infinite variety of solutions.  It is therefore necessary to select only those solutions which could represent limiting physical distributions. Since $t$ and $x$ do not explicitly appear in the resulting system, the solutions of interest will correspond to equilibrium states.
For the remainder of this paper, we shall restrict our attention
to the gaseous and liquid states at most.

\vspace{0.5em}
\noindent
\begin{theorem}\label{thm:gl}
    Recall from \eqref{6.1}, the functions $\phi_l$ satisfy those equations \eqref{7.1}
if and only if the $g_l$ satisfy
\begin{equation}
\frac{\partial g_l}{\partial \xi_i^\alpha}
  = -\frac{\rho}{D} \int_{-\infty}^{\infty}
    g_{l+1} \, \frac{\partial}{\partial \xi_i^\alpha}
    \bigl[ 1 - e^{-2A\Omega_l} \bigr] \, d\xi_{l},
  \qquad
  \int_{-\infty}^{\infty}
    g_{l+1} \, \frac{\partial}{\partial \xi_l^\alpha}
    \bigl[ 1 - e^{-2A\Omega_l} \bigr] \, d\xi_l = 0,
  \tag{7.3}\label{7.3}
\end{equation}
where
\[
  \Omega_l = \sum_{j=0}^{l-1} \Phi(\|\xi_l - \xi_j\|),
  \qquad
  \xi_0 = 0, \quad j = 0, \dots, l-1,
\]
\end{theorem}

\begin{proof}
    Passing to the limit in \eqref{6.33} and \eqref{6.36}, we see that $g_{kl}$ satisfy the left equation in \eqref{7.3}. Moreover, applying \eqref{6.22}
to both $C$ and $D$, we find that $g_{kl}$ is independent of $k$.

Finally, summing the right-hand side of \eqref{6.32} over $j=0,\dots,l-1$,
we obtain the right equation in \eqref{7.3}.
\end{proof}

This is the main simplification of Section \ref{Sectiongl}. The next lemma explains the structural consequences of \eqref{7.3} and shows what kind of recursive form one should expect for the functions $g_l$. The solution is given by a standard process named Mayer Cluster expansion.\cite{MayerMayer1940}
\begin{lemma}[Lemma 7.1]\label{lemma7.1}
If the functions $g_l$ are symmetric in the indices, are differentiable, and
differentiation under the integral sign is permitted, and if the $g_l$ satisfy
equation~\eqref{7.3}, then, for each $p$, we have
\begin{equation*}
g_{l\xi_i^\alpha}
= -\,\frac{\partial}{\partial \xi_i^\alpha}
  \sum^p_{q=1}\frac{\rho^p}{p!\,D^p}
  \int_{-\infty}^{\infty}
  h_{l p}(\xi; \eta; A)
  g_{l+p,\xi_i\alpha}(\xi, \eta; \rho, A)
  \,d\eta
\end{equation*}\[
+\frac{\rho^p}{p!\,D^p}
  \int_{-\infty}^{\infty}
  h_{l p,\xi_i\alpha}(\xi, \eta; A)
  g_{l+p}(\xi, \eta; \rho, A)
  \,d\eta,
  \label{7.4}\tag{7.4}
\]
where
\begin{align*}
\label{7.5}\tag{7.5}
h_{l p}(\xi_1, \dots, \xi_{l-1}; \eta_1, \dots, \eta_p; A)
&= H_{lp}(\xi_1, \dots, \xi_{l-1}; \eta_1, \dots, \eta_p; A)\,
   K_p(\eta_1, \dots, \eta_p; A),
   \\[0.5em]
K_p(\eta_1, \dots, \eta_p; A)
&= \exp\!\left\{
  -A \sum_{i,k=1}^p \Phi(\|\eta_i - \eta_k\|)
  \right\},
  \qquad K_0 = K_1 = 1, \\[0.5em]
H_{l p}(\xi_1, \dots, \xi_{l-1}; \eta_1, \dots, \eta_p; A)
&= \prod_{r=1}^p h_{l1}(\xi_1, \dots, \xi_{l-1}; \eta_r; A), \\[0.5em]
h_{l1}(\xi_1, \dots, \xi_{l-1}; \eta_1; A)
&= 1 - \exp\!\left\{
  -2A \sum_{i=1}^{l-1} \Phi(\|\xi_i - \eta_1\|)-2A\Phi(\|\xi_i - \eta_1\|)
  \right\},
  \qquad k_{10} = 1.
\end{align*}And $p$ can be any finite number independent of $l$. 
\end{lemma} 

For proof, see \cite{Morrey1955}{Section 6}.

The result of this lemma suggests trying a solution of \eqref{7.3} of the form
\begin{equation}
\label{7.9}\tag{7.9}
g_l(\xi_1, \dots, \xi_{l-1}; \rho, A)
= E_l(\rho; A)
  - \sum_{p=1}^\infty \frac{\rho^p}{p!\,D^p}
  \int_{-\infty}^{\infty}
  h_{l p} g_{l+p}\,d\eta,
\end{equation}
Let us come to the \textbf{conclusion} first, then move to the derivation: Our \textbf{only} solution $g_l$ which $g_l$ and $g_{l\xi_j^\alpha}$ are bounded and we can pass to the limit of equation \eqref{7.4} (w.r.t $p$) is equation \eqref{7.23}
\[
g_l(\xi_1, \dots, \xi_{l-1}; \rho, A)
= [E(\rho, A)]^l
  \sum_{q=0}^{\infty}
  f_{l q}(\xi_1, \dots, \xi_{l-1}; A)
  [-y(\rho, A)]^q/q!,\]
Where $E$ and $y$ has explicit formula in equation \eqref{7.22} and $f_lq$ has explicit formula in Section 7.

\bigskip
\noindent
\textbf{Derivation}(mostly based on induction)

Based on equation \eqref{7.9} we assume
\begin{equation}
\label{7.10}\tag{7.10}
g_l
= \sum_{p=0}^\infty \frac{\rho^p}{p!\,D^p}\Gamma_{l p}(\xi_1, \dots, \xi_{l-1}; A),
\qquad
E_l(\rho; A) = \sum_{p=0}^\infty C_{l p}(A)\frac{\rho^p}{p!\,D^p}.
\end{equation}
Since we require \eqref{6.2} to hold, we must have
\begin{equation}
\label{7.11}\tag{7.11}
\Gamma_{l 0} = 1,
\qquad
\lim_{\xi_i \to \infty,\,\xi_i - \xi_{l-1}\to \infty} \Gamma_{l p} = 0,
\quad p \ge 1,
\qquad
C_{1 0} = 1.
\end{equation}

Substituting \eqref{7.10} into \eqref{7.9} and equating coefficients of equal powers of
$\rho$, we obtain
\begin{equation}
\label{7.12}\tag{7.12}
\Gamma_{l p}
= C_{l p} - \sum_{q=1}^{p}
  \binom{p}{q}
  \int_{-\infty}^{\infty}
  h_{l q}(\xi_1, \dots, \xi_{l-1}; \eta_1, \dots, \eta_q; A)
  \Gamma_{l+q, p-q}(\xi, \eta; A)\,d\eta.
\end{equation}

From \eqref{7.12}, it follows by induction that
\begin{equation}
\label{7.13}\tag{7.13}
\Gamma_{l p}
= \sum_{q=0}^{p} (-1)^q \binom{p}{q} C_{l+q, p-q} f_{l q},
\qquad
f_{l 0} = 1,
\end{equation}
where
\[
f_{l p}(\xi_1, \dots, \xi_{l-1}; A)
:= \sum_{q=1}^{p} (-1)^{q-1}\binom{p}{q}
  \int_{-\infty}^{\infty}
  h_{l q}(\xi_1, \dots, \xi_{l-1}; \eta_1, \dots, \eta_q; A)
  f_{l+q, p-q}(\eta_1, \dots, \eta_q; A)\,d\eta.
\]

In \cite{Morrey1955}{Section 9}, we will know that
\begin{equation}
\label{7.14}\tag{7.14}
\lim_{\xi_i\to\infty,\;\xi_i-\xi_{l-1}\to\infty}
  f_{l p}(\xi_1, \dots, \xi_{l-1}; A)
  = d_{l p}(A),
\qquad
d_{l+1, p}
= \sum_{q=0}^p \binom{p}{q} d_{l, p-q}(A)d_{1 q}(A),
\qquad f_{1 p} = d_{1 p}.
\end{equation}

In Addition, from \cite{Morrey1955}{Section 7 and 8}, $f_{l p}$ has continuous derivatives with
respect to $\xi_1,\dots,\xi_{l-1}$, and that

\begin{equation}
\begin{aligned}
|f_{l p}(\xi_1, \dots, \xi_{l-1}; A)|
&\le [P(A)]^{l+p}\,[r(A)]^{p-1}\,l\,(l+p)^{p-1},\\
|f_{l\xi_i\alpha}(\xi_1, \dots, \xi_{l-1}; A)|
&\le [P(A)]^{l+p+1}\,[r(A)]^{p-1}\,p\,(l+1)\,(l+p)^{p-2},
\end{aligned}
\label{7.15}\tag{7.15}
\end{equation}

for all $\xi_1, \dots, \xi_{l-1}$, where $r(A)$ and $d(A)$ are as defined in
\cite{Morrey1955}{(8.9)} and $P(A)$ in \cite{Morrey1955}(7.27). We just quote the results to keep the structure consistent.

Since we require $\Gamma_{l p} \to 0$ as $\|\xi_i\|\to\infty$ or
$\|\xi_i - \xi_k\|\to\infty$ for $p\ge1$, we see from \eqref{7.13} and \eqref{7.14} that
\begin{equation}
\label{7.17}\tag{7.17}
\sum_{q=0}^{p}(-1)^q\binom{p}{q}d_{l q}C_{l+q, p-q} = 0, \qquad p \ge 1.
\end{equation}
Equations \eqref{7.17} and \eqref{7.11} determine the $C_{l p}$, so $E_l$ \textbf{uniquely} in terms of the
$d_{l p}$.

Now, using \eqref{7.10}, \eqref{7.11}, and \eqref{7.13}, we find formally that
\begin{equation}
\label{7.18}\tag{7.18}
g_l(\xi_1, \dots, \xi_{l-1}; \rho, A)
= \sum_{q=0}^{\infty}
  E_{l+q}(\rho; A)\,f_{l q}(\xi_1, \dots, \xi_{l-1}; A)\,
  (-\rho)^q/(q!\,D^q).
\end{equation}
Letting $\xi_i \to \infty$ and $\|\xi_i - \xi_k\| \to \infty$ in \eqref{7.18},
we obtain
\begin{equation}
\label{7.19}\tag{7.19}
\sum_{p=0}^{\infty}
  d_{l p}(A)E_{l+p}(\rho; A)(-\rho)^p/(p!\,D^p)
= 1.
\end{equation}

Now suppose we define $y(\rho; A)$ and $E(\rho; A)$ by
\begin{equation}
\label{7.20}\tag{7.20}
yG(y; A) = \rho/D, \qquad
y(\rho; A) = (\rho/D)E(\rho; A).
\end{equation}
where \begin{equation}
\label{7.16}\tag{7.16}
G(y; A) = \sum_{p=0}^{\infty} d_{l p}(A)(-y)^p/p!,
\quad
\|y\| \le [eP(A)r(A)]^{-1}.
\end{equation}
Accordingly, the bounds \eqref{7.15} holds for $d_{lp}$ and ensure that all
series for $G_l(y; A)$ below converge for $\|y\|$ sufficiently small, and also

\textbf{Remark:} The whole logic is: one can use $yG(y; A) = \rho/D$ and $G(y; A) = \sum_{p=0}^{\infty} d_{l p}(A)(-y)^p/p!$ to solve out $y(\rho,A)$ explicitly and then use $y$ to determine $E$. One can understand $y$ as the renormalization of $\rho/D$, and $E$ is the factor that corrects the renormalization and make \eqref{7.19} still be 1.

Then, from \eqref{7.15} and \eqref{7.16}, it follows that $y$ and $E$ are analytic in $\rho$
for each real $A>0$, and that
\begin{equation}
\label{7.21}\tag{7.21}
[\rho E/D]^l [G_l(\rho E/D, A)]^l
= (\rho E/D)^l G_l(\rho E/D, A)
= (\rho/D)^l,
\qquad l=1,2,\dots
\end{equation}
remembering \eqref{7.16}. Dividing \eqref{7.21} by $(\rho/D)^l$, we recover \eqref{7.19} with
$E_{l+q}$ replaced by $E^{l+q}$. Since the $C_{l p}$ and hence the $E_l$ are
uniquely determined by the $d_{l p}$, we have
\begin{equation}
\label{7.22}\tag{7.22}
E_l(\rho; A) = [E(\rho; A)]^l,
\end{equation}
where $E$ is defined by \eqref{7.20}. Therefore, we will have 
\begin{equation}
\label{7.23}\tag{7.23}
g_l(\xi_1, \dots, \xi_{l-1}; \rho, A)
= [E(\rho, A)]^l
  \sum_{q=0}^{\infty}
  f_{l q}(\xi_1, \dots, \xi_{l-1}; A)
  [-y(\rho, A)]^q/q!,
\end{equation}
which is the formal expression of $g_l$. The series being uniformly and absolutely convergent over the entire
$(\xi_1, \dots, \xi_{l-1})$--space for all $\rho$ such that $y$ satisfies \eqref{7.16}. $y$ is actually the expansion parameter measuring effective density/correlation strength.

Finally, using the bounds \eqref{7.15}, we see that
\begin{equation}
\label{7.24}\tag{7.24}
|g_l|
\le |P(A)E(\rho; A)|^l\,l
\sum_{p=0}^{\infty} (l+p)^{p-1}(P_T |y|)^p/p!
\le C_1\,l\,(PE)^l,
\qquad
C_1 = e^{-1} + (2\pi)^{-1/2}\sum_{p=1}^{\infty} p^{-3/2},
\end{equation}
and
\begin{equation*}
|g_{l\xi_i^\alpha}|
\le [\sigma(A)/\tau(A)](l+1)(PE)^l
\sum_{p=1}^{\infty}p(l+p)^{p-2}(p_T|y|)^p/p!
\le C_2 [\sigma/\tau](l+1)(PE)^l,
\qquad
C_2 = C_1 - e^{-1},
\end{equation*}
for $\|y\| \le (P_T e)^{-1}$.

Using and the definitions \eqref{7.5} with the results from \cite{Morrey1955}{Section 7,8}{\;(8.12)}, we obtain
\begin{align}
\label{7.26}\tag{7.26}
\int_{-\infty}^{\infty}
  |h_{l p}(\xi_1, \dots, \xi_{l-1}; \eta_1, \dots, \eta_p; A)|
  \,d\eta
&\le [P(A)]^{l+p+1}\,[r(A)]^{p}.
\end{align}

Using \eqref{7.24} and \eqref{7.26}, we can see that we may pass to the limit of $p$ in \eqref{7.4} if $\|y\| \le (P_T e)^{-1}$ and \eqref{7.9}(\eqref{7.23}) converges for such $y$. Hence our solution \eqref{7.23} is the \textbf{only} solution of \eqref{7.3} that $g_l$ and $g_{l\xi_j^\alpha}$ is bounded in such a way that the series \eqref{7.9} converges and we may pass to the limit of \eqref{7.4}. 

Such $g_l$ is unique w.r.t $\Omega_l$, the potential. The way to check is to assume $g_l$ and $g'_l$ and take difference. Since \eqref{7.3} is homogeneous, the difference $h_l=g_l-g_l'$ still satisfy \eqref{7.3} and tend to zero at infinity. One can indeed show $h_l=0$. For $g_l=1$, it is the solution only when $\Omega_l=0$, which stands for the potential vanishes and particles are independent.

In this section, we reduced the limiting equations for $\phi_l$ to a simpler system for the correlation functions $g_l$. Under the Gibbs-type equilibrium ansatz, the velocity dependence is fixed, so the main task becomes the determination of the spatial correlations. The reduced system \eqref{7.3} makes this possible and leads to a formal series solution for $g_l$. For sufficiently small density, this series converges and yields the equilibrium branch needed later in the derivation of the hydrodynamic equations.

\section{The final hydrodynamic equations and the entropy structure}\label{sectioneuler}

Finally, we can now explain how Morrey passes from the limiting hierarchy to the hydrodynamic equations. The point of this last step is to combine two ingredients obtained earlier in the paper. The first is the system of generalized balance laws derived in Section \ref{section2} from the reduced particle distributions. The second is the special equilibrium form of the limiting functions $\phi_l$ obtained in Sections \ref{section3}--\ref{Sectiongl}. Once these are put together, the macroscopic equations close and take the form of the compressible Euler system. The argument proceeds in three steps: (i) pass the balance laws to the limit and substitute the equilibrium $\phi_1,\phi_2$ to obtain a \emph{closed} Euler system; (ii) read off the equation of state $(p,\varepsilon)$ from the equilibrium correlations and express the pressure through the single normalization function $C$; and (iii) exhibit an entropy, built from $C$, that is transported by the flow.

\medskip
\noindent\textbf{Step (i): passage to the limit and closure.}
From Theorem \ref{Thm:limit}, by assuming Conjectures \ref{conj:pi}, \ref{conj:cell} and \ref{conj:deri}, the functions $\phi_{Nl}$ and their relevant derivatives converge uniformly to limiting functions $\phi_l$ \eqref{6.1}. Under this assumption, the balance laws of Section \ref{section2} for the finite-$N$ system pass to the limit, yielding equations for the density $\rho$, mean velocity $\bar u$, and energy density $e$, namely the $N\to\infty$ counterparts of \eqref{2.9}, \eqref{2.10}, and \eqref{2.18}:
\begin{align}
&\frac{\partial \rho}{\partial t}+\frac{\partial}{\partial x^\alpha}\big(\rho\,\bar u^\alpha\big)=0,
\tag{8.1a}\label{8.1a}\\[4pt]
&\frac{\partial}{\partial t}\big(\rho\,\bar u^\beta\big)
+\frac{\partial}{\partial x^\alpha}\big(\rho\,\bar u^\alpha \bar u^\beta+\tau_{1}^{\alpha\beta}\big)
=F_{1}^\beta(t;x),
\tag{8.1b}\label{8.1b}\\[4pt]
&e_{t}(t;x)
+\frac{\partial}{\partial x^\alpha}\big(\bar u^\alpha e\big)
+\partial_{x^\alpha}\omega_{1}^\alpha(t;x)
=W^\alpha_{,\alpha}(t;x),
\tag{8.1c}\label{8.1c}
\end{align}
where, in the limiting (equilibrium) distributions $\phi_1,\phi_2$, the kinetic stress, energy flux, and interaction terms are
\[
\tau_{1}^{\alpha\beta}(t;x)=\int_{\mathbb R^3}(u^\alpha-\bar u^\alpha)(u^\beta-\bar u^\beta)\,\phi_{1}(t;x,u)\,du,
\]
\[
F_{1}^\beta(t;x)
=\frac{K}{2D}\int
\Phi'\big(|\xi_1|\big)\frac{\xi_1^\beta}{|\xi_1|}\,
\partial_{x}\phi_{2}(t;x,u;\xi_1,v_1)\,du\,d\xi_1\,dv_1,
\]
\[
\omega_{1}^\alpha(t;x)
=\frac12\int_{\mathbb R^3}(u^\alpha-\bar u^\alpha)
\left[|u|^2\phi_{1}(t;x,u)
+\frac{K}{D}\int\Phi(|\xi_1|)\,\phi_{2}(t;x,u;\xi_1,v_1)\,d\xi_1\,dv_1\right]du,
\]
and $W^\alpha_{,\alpha}$ denotes the symmetrized interaction-work term, the limit of the right-hand side of \eqref{2.18}. (In Morrey's notation, $\tau^{\alpha\beta}$ and $\omega^\beta$ are the momentum-flux and energy-flux tensors; we keep his symbols.) At this stage the equations still contain $\tau_1^{\alpha\beta}$, $\omega_1^\alpha$, and the interaction terms $F_1^\beta,W^\alpha$, all expressed through the two-particle function $\phi_2$, so the system is not closed. The role of the concrete equilibrium construction \eqref{3.3} is precisely to identify these terms explicitly and reduce them to thermodynamic quantities.

We now substitute the equilibrium expressions for $\phi_1$ and $\phi_2$ \eqref{6.1}. The key simplification is that these tensors become \emph{isotropic}.

\begin{lemma}[Isotropy of the equilibrium stress]\label{lem:isotropy}
For the equilibrium one-particle distribution $\phi_1(x,u)=\rho\,(A/\pi)^{3/2}\exp\{-A\|u-\bar u\|^2\}$,
\[
\tau_1^{\alpha\beta}=\frac{\rho}{2A}\,\delta^{\alpha\beta},
\qquad
\delta^{\alpha\beta}:=\begin{cases}1,&\alpha=\beta,\\[2pt] 0,&\alpha\neq\beta.\end{cases}
\]
\end{lemma}

\begin{proof}
In the fluctuation variable $w=u-\bar u$, the weight $\phi_1$ depends on $w$ only through $\|w\|^2$, hence is invariant under the orthogonal group $O(3)$ acting on $w$. For $\alpha\neq\beta$ the integrand $w^\alpha w^\beta\,e^{-A\|w\|^2}$ is odd in $w^\alpha$, so $\tau_1^{\alpha\beta}=0$. For the diagonal entries, the reflection/permutation symmetry of $e^{-A\|w\|^2}$ forces $\tau_1^{11}=\tau_1^{22}=\tau_1^{33}$, each equal to the one-dimensional Gaussian variance
\[
\int_{\mathbb R^3}(w^\alpha)^2\,\rho\Big(\tfrac{A}{\pi}\Big)^{3/2}e^{-A\|w\|^2}\,dw=\frac{\rho}{2A}.
\]
Thus $\tau_1^{\alpha\beta}=\frac{\rho}{2A}\delta^{\alpha\beta}$.
\end{proof}

\noindent
More structurally, any symmetric rank-two tensor obtained by integrating a rotation-invariant weight against $w^\alpha w^\beta$ (or, for the interaction term, against $\xi^\alpha\xi^\beta/|\xi|$) must commute with every rotation; by Schur's lemma applied to the action of $SO(3)$ on symmetric $2$-tensors, the only such tensors are multiples of $\delta^{\alpha\beta}$. The interaction term $F_1^\beta$ in \eqref{8.1b} is, in the limit, the divergence of just such a tensor, so it too is isotropic and contributes a scalar (virial) pressure. Consequently the momentum flux collapses to $\rho\bar u^\alpha\bar u^\beta+p\,\delta^{\alpha\beta}$, and the energy flux $\omega_1^\alpha$ reduces to the pressure-work term $p\,\bar u^\alpha$. In Morrey's notation this is the closure
\[
\tau^{\alpha\beta}=-p\,\delta^{\alpha\beta},\qquad \omega^\beta=-p\,\bar u^\beta,
\]
matching the original derivation. The balance laws \eqref{8.1a}--\eqref{8.1c} therefore take the conservation form
\begin{equation}\label{euler-cons}\tag{8.2}
\rho_t + \partial_{x^\alpha}(\rho \bar u^\alpha)=0,
\qquad
\partial_t(\rho \bar u^\beta) + \partial_{x^\alpha}\!\big(\rho \bar u^\beta \bar u^\alpha + p\,\delta^{\beta\alpha}\big)=0,
\qquad
e_t + \partial_{x^\alpha}\!\big[(e+p)\,\bar u^\alpha\big]=0,
\end{equation}
where $e=\rho\big(\tfrac12|\bar u|^2+\varepsilon\big)$ is the total energy density and $\varepsilon$ the specific internal energy.

It is convenient to rewrite \eqref{euler-cons} in the advective variables $(\rho,\bar u,\varepsilon)$, since the equilibrium closure below is expressed naturally through $\varepsilon$ rather than the total energy $e$. Using continuity in the momentum balance gives the velocity equation in non-conservative form; subtracting the kinetic-energy identity (obtained by contracting the momentum equation with $\bar u^\beta$),
\[
\partial_t\!\Big(\tfrac12\rho|\bar u|^2\Big)+\partial_{x^\alpha}\!\Big(\tfrac12\rho|\bar u|^2\,\bar u^\alpha\Big)
= -\,\bar u^\beta\,p_{x^\beta},
\]
from the total-energy equation, and using continuity once more, isolates the internal part:
\[
\rho\big(\varepsilon_t+\bar u^\alpha \varepsilon_{x^\alpha}\big) = -\,p\,\partial_{x^\alpha}\bar u^\alpha
= p\,\rho^{-1}\big(\rho_t+\bar u^\alpha\rho_{x^\alpha}\big).
\]
Hence \eqref{euler-cons} is equivalent to
\begin{equation}\label{euler-adv}\tag{8.3}
\rho_t + \partial_{x^\alpha}(\rho \bar u^\alpha)=0,
\qquad
\bar u^\beta_t + \bar u^\alpha \bar u^\beta_{x^\alpha}
   = - \rho^{-1} p_{x^\beta},
\qquad
\varepsilon_t + \bar u^\alpha \varepsilon_{x^\alpha}
   = p\,\frac{\rho_t+\bar u^\alpha \rho_{x^\alpha}}{\rho^2}.
\end{equation}
This is Morrey's Euler-level hydrodynamic system: conservation of mass, momentum balance with isotropic pressure, and the internal-energy law in the form ``$d\varepsilon=-p\,d(1/\rho)$''. No viscous or heat-conduction terms appear, in agreement with Morrey's interpretation that his argument captures only the leading, equilibrium-scale behavior of the fluid.

\medskip
\noindent\textbf{Step (ii): the equation of state.}
It remains to determine $p$ and $\varepsilon$ as functions of the thermodynamic state. Both are read directly from the equilibrium distributions \eqref{6.1}: $\varepsilon$ is the mean energy per particle in the local Gibbs state, and $p$ is the trace of the (kinetic plus virial) stress carried by $\phi_1$ and $\phi_2$. The energy density \eqref{2.10} splits into a kinetic part, governed by the Gaussian velocity factor in \eqref{6.1}, and a potential part, governed by the pair correlation $g_2$. The Gaussian factor has variance set by the inverse-temperature parameter, so the kinetic contribution to the specific internal energy is the equipartition value $\tfrac{3}{4B}$ (three translational degrees of freedom). Averaging $\tfrac12\Phi$ against the equilibrium pair distribution $e^{-2KA\Phi}g_2$ gives the potential part, so that
\begin{equation}\label{eos-eps}\tag{8.4}
\varepsilon
 = \frac{3}{4B}
   + \frac{K\rho}{2D}
     \int \Phi(|\xi_1|)\,
     e^{-2KA\Phi(|\xi_1|)}\, g_2(\xi_1;\rho,A)\,d\xi_1,
\end{equation}
where the first term is the ideal (kinetic) part and the second the interaction correction. Here $B=\tfrac{3}{4\varepsilon}$ is the bare inverse-temperature parameter of \eqref{5.16}, while $A=B\big(1-\tfrac{B}{C}\tfrac{\partial C}{\partial b}\big)$ from \eqref{6.5} is the dressed parameter appearing in the Gibbs exponent; the two coincide only in the non-interacting limit. The same computation applied to the momentum flux gives
\begin{equation}\label{eos-p}\tag{8.5}
p
 = \frac{\rho}{2A}
   - \frac{K\rho^2}{6D}
     \int |\xi_1|\,\Phi'(|\xi_1|)\,
     e^{-2KA\Phi(|\xi_1|)}\, g_2(\xi_1;\rho,A)\,d\xi_1,
\end{equation}
the first term being the kinetic pressure and the second the standard virial integral $-\tfrac16\rho^2\!\int r\,\Phi'(r)\,g_2$. Equations \eqref{eos-eps}--\eqref{eos-p} show that the equation of state is not imposed abstractly but extracted from the limiting particle distribution itself.

This is the central structural point of the construction, and it is worth pausing on it. In the kinetic (dilute-gas) route to hydrodynamics discussed in Section~\ref{subsec:motivation}, the second terms in \eqref{eos-eps}--\eqref{eos-p} are absent: with no interaction correction, $A=B$ and one is left with $\varepsilon=\tfrac{3}{4B}$, $p=\tfrac{\rho}{2B}$, hence the rigid ideal-gas law
\[
p=\tfrac23\rho\varepsilon=(\gamma-1)\rho\varepsilon,\qquad \gamma=\tfrac53\ \text{fixed.}
\]
In Morrey's setting the interaction terms survive: the pressure depends on the potential $\Phi$ through the virial integral and on the correlation $g_2$, so the equation of state is no longer pinned to a single value of $\gamma$ but varies with the microscopic interaction. This is precisely the flexibility anticipated in Section~\ref{subsec:motivation}, and we now make it quantitative by tying $p$ to the normalization function $C$.

To do so we return to $C(\rho,b)$ from Theorem \ref{Thm:limit}, which carries the normalization information of the equilibrium family and hence reflects its thermodynamic structure. The asymptotic computation of Section~\ref{Sectiongl} (specifically the density-derivative of $\log C$, cf.\ \eqref{6.21}) reproduces exactly the virial integral in \eqref{eos-p}:
\begin{equation}\label{C-virial}\tag{8.6}
\frac{\rho C_\rho}{C}
 = \frac{2}{9}KA\Bigl(\frac{\rho}{D}\Bigr)
   \int |\xi_1|\,\Phi'(|\xi_1|)\,
   e^{-2KA\Phi(|\xi_1|)}\, g_2(\xi_1;\rho,A)\,d\xi_1.
\end{equation}
Substituting \eqref{C-virial} into \eqref{eos-p} to eliminate the integral, and using $A=B\big(1-\tfrac{B}{C}C_b\big)$ from \eqref{6.5}, gives the equation of state in closed form,
\begin{equation}\label{eos-C}\tag{8.7}
p = \frac{\rho}{2A} - \frac{3}{4A}\,\frac{\rho^2 C_\rho}{C}.
\end{equation}
Thus the pressure is expressed \emph{directly through the single function $C$}: the kinetic pressure $\rho/(2A)$ is corrected by a term $-\tfrac{3}{4A}\rho^2 C_\rho/C$ measuring how the equilibrium normalization responds to changes in density. This is the transition from the correlation description to thermodynamics, and it shows that the closure is controlled by $C$ --- and hence, through $C$, by the interaction potential $\Phi$ --- rather than by a prescribed adiabatic exponent.

\medskip
\noindent\textbf{Step (iii): the entropy.}
The same function $C$ generates the entropy, and here the choice of $\varepsilon$ as working variable pays off. Since the kinetic part of \eqref{eos-eps} gives $\varepsilon=\tfrac{3}{4B}$, equivalently $B=\tfrac{3}{4\varepsilon}$, the natural object to examine is the differential $d\varepsilon-p\,d(1/\rho)=d\varepsilon-p\,\rho^{-2}\,d\rho$ on the left of the internal-energy law \eqref{euler-adv}: by the first law this is $T\,dS$ along the flow, so its integrating factor is the inverse temperature $A=1/(2kT)$. Computing this combination from \eqref{eos-eps} and \eqref{eos-C} yields
\begin{equation}\label{dS-expand}\tag{8.8}
A\Bigl(d\varepsilon - p\,\frac{d\rho}{\rho^2}\Bigr)
 = -\frac{3}{4}\frac{dB}{B}
   + \frac{3}{4}\frac{C_B\,dB}{C}
   - \frac{1}{2}\frac{d\rho}{\rho}
   + \frac{3}{4}\frac{C_\rho\,d\rho}{C},
\end{equation}
matching Morrey's (13.10). The right-hand side is an exact differential: it is $d$ of
\begin{equation}\label{entropy}\tag{8.9}
S = -\frac12 \log \rho + \frac34 \log \varepsilon + \frac34 \log C + \text{const.},
\end{equation}
where we used $-\tfrac34\log B=\tfrac34\log\varepsilon+\text{const.}$ (since $B=\tfrac{3}{4\varepsilon}$) to convert the $dB/B$ term. With this definition \eqref{dS-expand} reads
\[
A\Bigl(d\varepsilon - p\,\frac{d\rho}{\rho^2}\Bigr)=dS,
\]
so $A$ is the integrating factor and $S$ the corresponding entropy. The left-hand side of the internal-energy law \eqref{euler-adv} is precisely $\varepsilon_t+\bar u^\alpha\varepsilon_{x^\alpha}-p\rho^{-2}(\rho_t+\bar u^\alpha\rho_{x^\alpha})$, so dividing by $A$ and substituting \eqref{dS-expand} gives the transport law
\[
S_t + \bar u^\alpha S_{x^\alpha}=0.
\]
Hence along the limiting flow the entropy is conserved --- exactly the entropy law expected for the compressible Euler equations in the absence of dissipative effects.

This calculation clarifies that $C(\rho,b)$ is not merely an auxiliary normalization factor from the equilibrium construction. Once the hydrodynamic limit is taken, $C$ enters the equation of state \eqref{eos-C} and contributes directly to the entropy \eqref{entropy}: the same object that controls the convergence of the marginals in Theorem~\ref{Thm:limit} also fixes the pressure law and the entropy. This is what allows Morrey's closure to accommodate a potential-dependent equation of state rather than a fixed polytropic one. Finally, Morrey records qualitative properties of $C$ --- convexity in $b$ for $b\ge0$, and the bounds
\[
1-\tfrac{2}{3}\phi(h)\le C(\rho,b)\le 1+\tfrac{2}{3}KLb
\]
of Lemma~\ref{lem:bounds}, with $C(\rho,b)$ positive for $\rho$ sufficiently small. These suggest that the region where $C(\rho,b)>0$ is, in Morrey's words, \emph{probably} the largest set to which the equilibrium family --- and hence the hydrodynamic description --- can be continued by analytic continuation.

\bibliographystyle{plain}
\bibliography{ref}

@book{CIP1994,
  author    = {Cercignani, Carlo and Illner, Reinhard and Pulvirenti, Mario},
  title     = {The Mathematical Theory of Dilute Gases},
  series    = {Applied Mathematical Sciences},
  volume    = {106},
  publisher = {Springer-Verlag},
  address   = {New York},
  year      = {1994},
  isbn      = {978-0-387-94294-0},
  doi       = {10.1007/978-1-4419-8524-8}
}

@book{MayerMayer1940,
  author    = {Joseph E. Mayer and Maria Goeppert Mayer},
  title     = {Statistical Mechanics},
  publisher = {John Wiley \& Sons},
  address   = {New York},
  year      = {1940},
}

@incollection{DeMasiEtAl1984,
  author    = {De Masi, A. and Ianiro, N. and Pellegrinotti, A. and Presutti, E.},
  title     = {A Survey of the Hydrodynamical Behavior of Many-Particle Systems},
  booktitle = {Studies in Statistical Mechanics},
  volume    = {11},
  editor    = {Montroll, E. W. and Lebowitz, J. L.},
  publisher = {North-Holland},
  address   = {Amsterdam},
  year      = {1984},
}

@book{BornGreen1949,
  author    = {Born, Max and Green, Herbert S.},
  title     = {A General Kinetic Theory of Liquids},
  publisher = {Cambridge University Press},
  address   = {Cambridge},
  year      = {1949}
}

@article{Morrey1955,
  author  = {Morrey, Charles B., Jr.},
  title   = {On the Derivation of the Equations of Hydrodynamics from Statistical Mechanics},
  journal = {Communications on Pure and Applied Mathematics},
  volume  = {8},
  pages   = {279--326},
  year    = {1955}
}

@incollection{Lanford1975,
  author    = {Lanford, Oscar E., III},
  title     = {Time Evolution of Large Classical Systems},
  booktitle = {Dynamical Systems, Theory and Applications},
  editor    = {Moser, J.},
  series    = {Lecture Notes in Physics},
  volume    = {38},
  pages     = {1--111},
  publisher = {Springer},
  address   = {Berlin, Heidelberg},
  year      = {1975},
  doi       = {10.1007/3-540-07171-7_1}
}

@article{Caflisch1980,
  author  = {Caflisch, Russel E.},
  title   = {The Fluid Dynamic Limit of the Nonlinear Boltzmann Equation},
  journal = {Communications on Pure and Applied Mathematics},
  volume  = {33},
  number  = {5},
  pages   = {651--666},
  year    = {1980},
  doi     = {10.1002/cpa.3160330506}
}

@article{BardosGolseLevermore1991,
  author  = {Bardos, Claude and Golse, Fran{\c c}ois and Levermore, C. David},
  title   = {Fluid Dynamic Limits of Kinetic Equations. I. Formal Derivations},
  journal = {Journal of Statistical Physics},
  volume  = {63},
  number  = {1--2},
  pages   = {323--344},
  year    = {1991},
  doi     = {10.1007/BF01026608}
}

@book{GallagherSaintRaymondTexier2014,
  author    = {Gallagher, Isabelle and Saint-Raymond, Laure and Texier, Benjamin},
  title     = {From Newton to Boltzmann: Hard Spheres and Short-Range Potentials},
  publisher = {European Mathematical Society},
  series    = {Zurich Lectures in Advanced Mathematics},
  year      = {2014},
  doi       = {10.4171/129}
}

@misc{DengHaniMa2025,
  author       = {Deng, Yu and Hani, Zaher and Ma, Xiao},
  title        = {Hilbert's Sixth Problem: Derivation of Fluid Equations via Boltzmann's Kinetic Theory},
  year         = {2025},
  eprint       = {2503.01800},
  archivePrefix= {arXiv},
  primaryClass = {math.AP},
  url          = {https://arxiv.org/abs/2503.01800}
}

@book{Khinchin1949,
  author    = {Khinchin, Aleksandr Ya.},
   title     = {Mathematical Foundations of Statistical Mechanics},
 publisher = {Dover Publications},
   address   = {New York},
   year      = {1949},
   note      = {Translated from the Russian by G.~Gamow}
 }

@book{Ruelle1969,
  author    = {Ruelle, David},
  title     = {Statistical Mechanics: Rigorous Results},
  publisher = {W.~A. Benjamin},
  address   = {New York},
  year      = {1969},
  note      = {Reprinted by World Scientific, Singapore, 1999}
}

\appendix
\section{A priori bounds}\label{sectionAbounds}
 Throughout the paper, one can establish certain estimations, especially on the summation of $\Phi$. They are important both in the convergence of $\phi_{Nl}$ and in estimating the Fourier transform \eqref{3.6}. 
\begin{lemma}
    [Existence of Lower Bound Functions]\label{lem:psi_chi}
There exist functions $\Psi(h)$ and $\mathcal{X}(h)$ with the following properties:
\begin{enumerate}[label=(\alph*)]
    \item If $y_1, \dots, y_n$ are any vectors with all pairwise distances 
    $|y_i - y_k| \ge h$, then for each $j=1,\dots,n$,
    \begin{equation*}
    \sum_{k\ne j} \Phi\!\left(|y_k - y_j|\right) \ge \Psi(h),
    \end{equation*}
    where $\Psi(h)$ is an increasing function.
    
    \item If also for some fixed $j_0,k_0$ we have $|y_{j_0} - y_{k_0}| = h$, then
    \begin{equation*}
    \sum_{k\ne j}^n \Phi\!\left(|y_k-y_j|\right) 
    \ge \Phi(h) + \Psi(h) \equiv \mathcal{X}(h),
    \qquad j = j_0 \text{ or } k_0,
    \end{equation*}
    and 
    \[
    \lim_{h \to 0^+} \mathcal{X}(h) = +\infty.
    \]
\end{enumerate}
    
\end{lemma}

\begin{lemma}
    \label{thm:lower_bound}
There exists a constant $L \ge 0$ such that
\begin{equation}
\sum_{i\neq k}^n \alpha_{ik}\, \Phi\!\left(|y_j - y_k|\right)
\ge -\, n L,\tag{A.1}
\label{eq:lower_bound}
\end{equation}
for any vectors $y_1, \dots, y_n$, where $\alpha_{ik} = \alpha_{ki}$ and each $\alpha_{ik}$ is either $0$ or $1$.
\end{lemma}

\textbf{Remark :} The proofs, which rely on geometric estimates of particle separations and monotonicity properties of $\Phi$, are omitted here for brevity. They ensure that interaction sums involving $\Phi(r)$ cannot decrease indefinitely, establishing a uniform lower bound $L$ independent of particle number.

\begin{theorem}[A.1]\label{thm:LNbound}
There exists a constant \(L_1\), independent of \(N\), such that for every sequence
of distinct vectors \(y_1,\dots,y_N\),
\begin{equation}
N^{-1}\sum_{i,k=1}^{N} \|y_i-y_k\|\,\big|\Phi'(\|y_i-y_k\|)\big|
\ \le\ 
(\delta+1)\,N^{-1}\sum_{i,k=1}^{N}\Phi(\|y_i-y_k\|)\ +\ L_1.
\label{A.2}\tag{A.2}
\end{equation}
\end{theorem}

\begin{proof}
Define
\begin{equation}
\Phi_1(r):=(\delta+1)\,\Phi(r)-r\,\big|\Phi'(r)\big|.
\label{A.3}\tag{A.3}
\end{equation}
We know \(\Phi_1\) satisfies assumption \eqref{1.2}.
Hence by Lemma \ref{thm:lower_bound} there exists \(L_1\ge 0\) (independent of \(N\) and of the configuration)
such that
\begin{equation}
N^{-1}\sum_{i,k=1}^{N}\Phi_1(\|y_i-y_k\|)\ \ge\ -\,L_1.
\label{A.4}\tag{A.4}
\end{equation}
Expanding \(\Phi_1\) from \eqref{A.3} and rearranging yields
\[\label{A.5}\tag{A.5}
N^{-1}\sum_{i,k} \|y_i-y_k\|\,\big|\Phi'(\|y_i-y_k\|)\big|
\ \le\ 
(\delta+1)\,N^{-1}\sum_{i,k}\Phi(\|y_i-y_k\|)\ +\ L_1,
\]
which is \eqref{A.2}.
\end{proof}

\textbf{Remark:}
Inequality \eqref{A.2} bounds the \emph{weighted} force by the potential energy plus a constant.
This is the quantitative device needed later to convert microscopic force sums to macroscopic fluxes.

\begin{theorem}[A.2]\label{thm:A.2}
If, at some instant, the total energy of the system is \(E_N\), then \(E_N\) is constant for all times and
\begin{equation}
m_ N\sum_{i=1}^{N}\|u_i\|^2\ \le\ 2E_N + LNm_N,
\label{A.6}\tag{A.6}
\end{equation}
where \(L\) is the constant in Theorem~5.1 (applied to \(\Phi\)).

If at time \(t_0\),
\begin{equation}
m_ N\sum_{i=1}^{N}\|x_i(t_0)\|^2\ =:\ C_N,
\label{A.7}\tag{A.7}
\end{equation}
then for all \(t\),
\begin{equation}
m_ N\sum_{i=1}^{N}\|x_i(t)\|^2
\ \le\ \Big(C_N^{1/2} + (2E_N+LM_N)^{1/2}\,|t-t_0|\Big)^2.
\label{A.8}\tag{A.8}
\end{equation}

\end{theorem}

\begin{proof}
Energy conservation follows from the microscopic equations of motion (Section~2). One need to use the scaling in section 2 to see $K$ here.
Write
\[
E_N \;=\; \frac{1}{2}m_ N \sum_{i=1}^{N}\|u_i\|^2\ +\
\frac{1}{2}m_ N\sum_{i\neq k} \Phi(\|x_i-x_k\|/\epsilon),
\]
and involve the lower bound from Theorem~5.1 applied to the \(\Phi\)-sum to obtain
\[
\frac{1}{2}m_ N \sum_i \|u_i\|^2\ \le\ E_N + \frac{1}{2}Nm_ NL,
\]
which gives us \ref{A.6}.

 Let \(f(t)=m_ N\sum_i \|x_i(t)\|^2\).
Then \(f'(t)=2m_ N\sum_i x_i\cdot u_i\).
By Cauchy–Schwarz and \ref{A.6},
\[
|f'(t)|\ \le\ 2\Big(m_ N\sum_i \|x_i\|^2\Big)^{1/2}
\Big(m_ N\sum_i \|u_i\|^2\Big)^{1/2}
\ \le\ 2\sqrt{f(t)}\, (2E_N+LNm_N)^{1/2}.
\]
Thus \(\dfrac{d}{dt}\sqrt{f(t)}\ \le\ (2E_N+LNm_N)^{1/2}\).
Integrating from \(t_0\) gives \(\sqrt{f(t)}\le \sqrt{f(t_0)}+(2E_N+LNm_N)^{1/2}|t-t_0|\),
which is \eqref{A.8}.
\end{proof}

\textbf{Remark:}
Theorem~\ref{thm:A.2} is the key uniform kinetic-energy bound.
It says that even if particles momentarily cluster (raising potential energy),
the mean kinetic energy cannot exceed a linear function of total energy and mass.

Estimate \eqref{A.8} shows that the second moment of positions grows at most $t^2$ in time, with linear rate controlled by the energy-mass bound.
This is the time-scale control needed for compactness of transforms below.

\section{Roadmap of the Argument}\label{sec:roadmap}

The derivation of the compressible Euler equations from
$N$-particle mechanics involves a long chain of constructions,
estimates, and limiting arguments of varying degrees of rigour.
This section provides a complete map of that chain before the
technical details begin.

\subsection{What Is and Is Not Proved}

To prevent any ambiguity about the status of the main result,
we state explicitly what the paper establishes.

\begin{description}[leftmargin=0pt, itemsep=4pt,
                    font=\normalfont\itshape]

\item[What is fully proved.]
Given Assumption \ref{Phi}, the a priori bounds of Theorems \ref{thm:LNbound}-\ref{thm:A.2}
hold for all time.
Given Assumptions \ref{Phi}--\ref{assump:R_N}, the Fourier transforms
$F(D_N^\gamma)$ for $\gamma=1,2,3,4$ are compact
(Theorem~2), and the partition functions $C_{Nkl}$
converge subsequentially to a convex limit $C(\rho,b)$
with controlled derivative asymptotics (Lemmas~1--4).
Given Conjectures \ref{conj:cell}--\ref{conj:deri} and the non-degeneracy condition
$C(\rho_0,B_0)>0$, the marginals $\phi_{Nl}$ converge
to the Gibbs form (Theorem \ref{Thm:limit}).
Finally, the Gibbs form is the unique bounded solution
to the equilibrium hierarchy (Theorem \ref{thm:gl}, Lemma \ref{lemma7.1}, \S\ref{Sectiongl})
for sufficiently small density.

\item[What requires open conjectures.]
The identification of the limiting time derivatives with
macroscopic field derivatives (Assumption \ref{assump:ergodic} and Conjecture \ref{conj:pi}),
the validity of the one-cell approximation (Conjecture \ref{conj:cell}),
the $x$-derivative bounds for $g_{Nkl}$ (Conjecture \ref{conj:deri}),
and the full Fourier compactness for the energy observable
(Conjecture \ref{conj1}) are all unproved.
The compressible Euler system of Section~\ref{sectioneuler}
is therefore a \emph{formal} derivation, contingent on the
resolution of all four conjectures.

\item[Comparison with modern results.]
The remaining gap is substantial.
Lanford's derivation of the Boltzmann equation~\cite{Lanford1975}
achieves rigour by restricting to short time and hard spheres.
The entropy-method approach of Yau and Varadhan for
hydrodynamic limits of lattice gases~\cite{DeMasiEtAl1984}
replaces the ergodic hypothesis with quantitative entropy
estimates. Morrey's 1955 program predates both frameworks
and should be understood as an early, ambitious attempt
to identify the correct limiting structure, not as
a complete proof.

\end{description}

Figure~\ref{fig:roadmap} shows how every ingredient of the
argument feeds into the final result.
Solid arrows denote proved logical implications;
dashed arrows denote dependencies that are required but
rest on an unproved conjecture.

\begin{figure}[p]   
\centering

\begin{tikzpicture}[
  proved/.style = {
    rectangle, draw=blue!65!black, fill=blue!6, thick,
    rounded corners=4pt, text width=2.85cm,
    minimum height=0.9cm, align=center, font=\scriptsize
  },
  assumed/.style = {
    rectangle, draw=green!55!black, fill=green!6, thick,
    rounded corners=4pt, text width=2.85cm,
    minimum height=0.9cm, align=center, font=\scriptsize
  },
  openconj/.style = {
    rectangle, draw=red!60!black, fill=red!4,
    thick, dashed, rounded corners=4pt, text width=2.6cm,
    minimum height=0.9cm, align=center, font=\scriptsize
  },
  final/.style = {
    rectangle, draw=orange!75!black, fill=orange!10,
    very thick, rounded corners=5pt, text width=4.4cm,
    minimum height=1.15cm, align=center, font=\scriptsize\bfseries
  },
  sa/.style = {   
    -{Stealth[length=5pt, width=3.5pt]},
    thick, draw=gray!65
  },
  da/.style = {   
    -{Stealth[length=5pt, width=3.5pt]},
    thick, draw=red!55,
    dash pattern=on 4.5pt off 2pt
  },
]


\node[assumed]  (A1) at ( 0.0,  0.0)
  {\textbf{Assumption 1}\\Potential $\Phi$\\conditions};

\node[assumed]  (A2) at ( 4.8,  0.0)
  {\textbf{Assumption 2}\\Cell decomposition};

\node[assumed]  (A3) at ( 9.6,  0.0)
  {\textbf{Assumption 3}\\Ergodic hypothesis};

\node[proved]   (AP) at ( 0.0, -2.2)
  {\textbf{Thms 5--6, App.}\\A priori energy\\and position bounds};

\node[proved]   (MN) at ( 4.8, -2.2)
  {\textbf{Section 3}\\Constraint mfd.\ $M_N$\\measure $\mu_{M_N}$};

\node[proved]   (BL) at ( 9.6, -2.2)
  {\textbf{Section 2}\\Formal balance\\laws for $\rho_N,u_N,e_N$};

\node[proved]   (T2) at ( 0.0, -4.5)
  {\textbf{Theorem 2}\\Fourier compact.\\$\gamma = 1,2,3,4$};

\node[openconj] (C1) at ( 2.7, -4.5)
  {\textbf{Conj.\ 1}\\Extend to\\$\gamma = 5$};

\node[proved]   (S5) at ( 5.5, -4.5)
  {\textbf{Section 5}\\Finite-$N$ formula\\for $\phi_{Nl}$};

\node[openconj] (C2) at ( 8.7, -4.5)
  {\textbf{Conj.\ 2}\\Balance laws hold\\for $\mu_{M_N}$ mod $r_N$};

\node[proved]   (L14) at ( 0.0, -6.8)
  {\textbf{Lemmas 1--4}\\Bounds \& subseq.\\ conv.\ of $C_{Nkl}$};

\node[openconj] (C3)  at ( 4.0, -6.8)
  {\textbf{Conj.\ 3}\\One-cell limit\\$=$ general limit};

\node[openconj] (C4)  at ( 7.8, -6.8)
  {\textbf{Conj.\ 4}\\$x$-derivatives of\\$g_{Nkl}$ bounded};

\node[proved]   (T3)  at ( 4.2, -9.0)
  {\textbf{Theorem 3}\\$\phi_{Nl}\to\phi_l$\\(Gibbs form)};

\node[proved]   (T4) at ( 1.8, -11.2)
  {\textbf{Theorem 4}\\Limiting BBGKY\\$\Leftrightarrow$ system for $g_l$};

\node[proved]   (GL) at ( 7.0, -11.2)
  {\textbf{Lemma 5 + \S7}\\Explicit series for\\$g_l$ (small $\rho$)};

\node[final]    (EU) at ( 4.6, -13.8)
  {\textbf{Section 8}\\Compressible Euler Equations\\
   $+$\; Entropy: $S_t + \bar{u}^{\alpha}S_{x^{\alpha}} = 0$};


\draw[sa] (A1.south) -- (AP.north);

\draw[sa] (A1.south) .. controls +(0.8,-0.5) and +(-1.2, 0.5) ..
          (MN.north west);

\draw[sa] (A2.south) -- (MN.north);

\draw[sa] (A3.south) -- (C2.north);

\draw[sa] (AP.south) -- (T2.north);

\draw[sa] (A1.south) .. controls +(-1.0,-2.2) and +(-1.0, 0.5) ..
          (L14.north west);

\draw[sa] (MN.east) -- (BL.west);

\draw[sa] (MN.south) -- (S5.north);

\draw[da] (C1.north west) .. controls +(-0.2, 0.35) and +(0.25,-0.25) ..
          (T2.north east);

\draw[sa] (S5.south) .. controls +(0,-0.4) and +(0.5, 0.4) ..
          (C3.north east);

\draw[sa] (L14.south east) -- (T3.north west);

\draw[da] (C3.south) -- (T3.north);

\draw[da] (C4.south) .. controls +(0,-0.5) and +(1.6, 0.3) ..
          (T3.east);

\draw[da] (C2.south) .. controls +(0,-0.8) and +(1.2, 1.2) ..
          (T3.north east);

\draw[sa] (T3.south west) -- (T4.north east);

\draw[sa] (T3.south east) -- (GL.north west);

\draw[sa] (T4.south) -- (EU.north west);

\draw[sa] (GL.south) -- (EU.north east);

\draw[sa] (BL.south) .. controls +(0,-5.5) and +(3.2, 0.3) ..
          (EU.east);

\draw[da] (C2.south) .. controls +(0,-5.0) and +(2.8, 0.2) ..
          (EU.east);

\draw[sa] (A3.south) .. controls +(1.2,-9.2) and +(3.0, 0.6) ..
          (EU.north east);

\begin{scope}[shift={(0.0,-15.6)}]
  \filldraw[fill=gray!4, draw=gray!40, rounded corners=4pt]
    (-0.4,-0.55) rectangle (10.0, 0.55);

  \filldraw[fill=blue!6,  draw=blue!65!black,  thick, rounded corners=2pt]
    (0.0,-0.28) rectangle (0.7, 0.28);
  \node[font=\scriptsize, anchor=west] at (0.80, 0) {Proved result};

  \filldraw[fill=green!6, draw=green!55!black, thick, rounded corners=2pt]
    (2.9,-0.28) rectangle (3.6, 0.28);
  \node[font=\scriptsize, anchor=west] at (3.70, 0) {Structural assumption};

  \filldraw[fill=red!4, draw=red!60!black, thick, dashed, rounded corners=2pt]
    (6.3,-0.28) rectangle (7.0, 0.28);
  \node[font=\scriptsize, anchor=west] at (7.10, 0) {Open conjecture};

  \draw[sa, gray!65] (0.0,-0.50) -- (0.7,-0.50);
  \node[font=\scriptsize, anchor=west] at (0.80,-0.50) {Proved implication};
  \draw[da]          (3.7,-0.50) -- (4.4,-0.50);
  \node[font=\scriptsize, anchor=west] at (4.50,-0.50) {Unproved dependency};
\end{scope}

\end{tikzpicture}

\caption{%
  Logical dependency diagram for the derivation of the Euler equations.
  \textbf{All four open conjectures} (red dashed boxes) must be resolved
  for Section~\ref{sectioneuler} to constitute a rigorous proof.
  The three structural assumptions (green boxes) supply the physical
  content of the argument and cannot be eliminated within Morrey's framework.%
}
\label{fig:roadmap}
\end{figure}
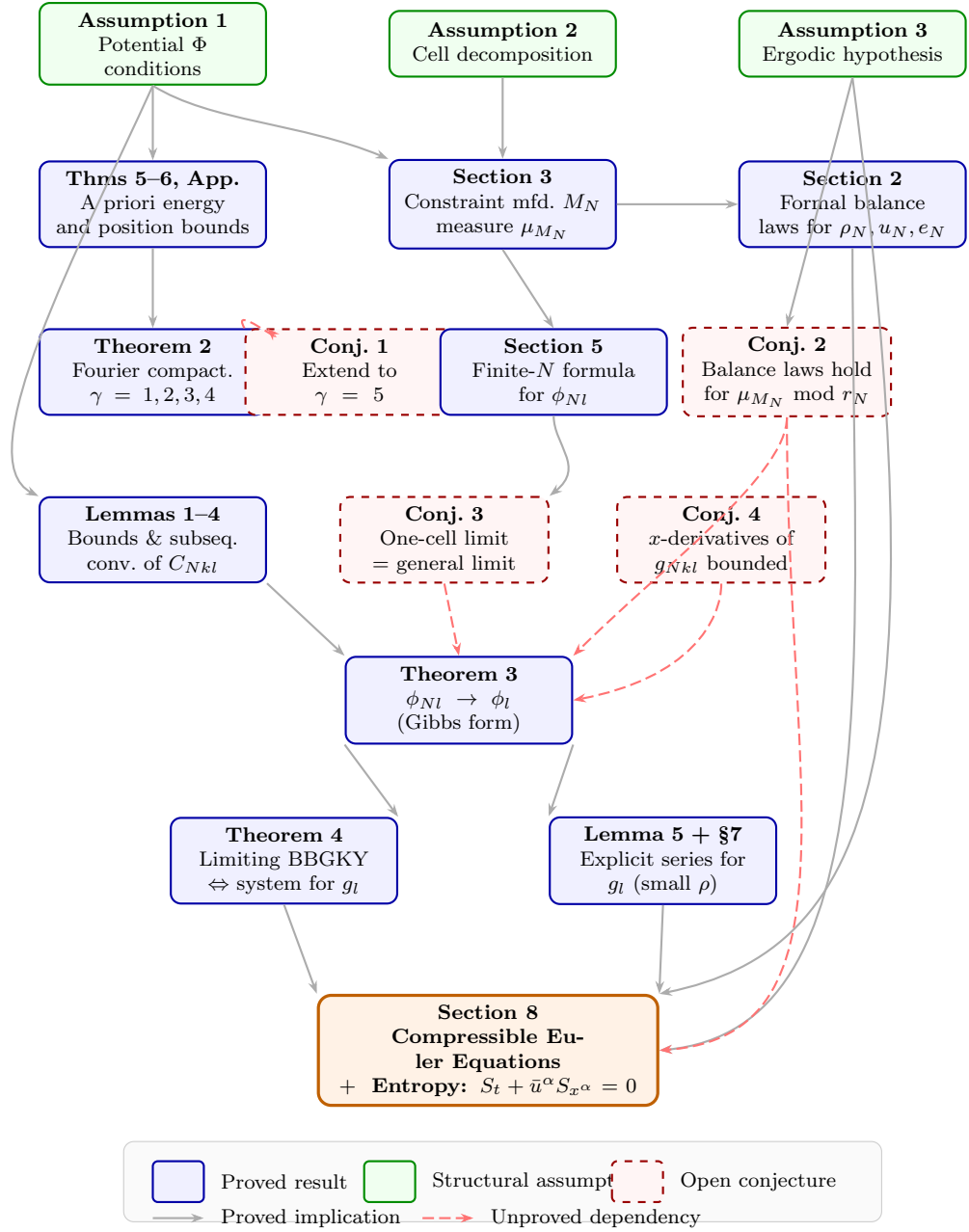

\end{document}